\documentclass[11pt,twoside,reqno,psamsfonts]{amsart}

\usepackage[left=2.9cm,top=3cm,right=2.9cm]{geometry}
\usepackage[colorlinks = true, citecolor = black, linkcolor = black, urlcolor = black, pdfstartview=FitH]{hyperref}  
\usepackage[pdftex]{graphicx}
\usepackage[utf8]{inputenc}
\usepackage{microtype, float}
\usepackage{amsmath, verbatim}
\usepackage{amssymb,mathtools}
\usepackage{epstopdf}
\usepackage{epsfig}
\usepackage{mathscinet}

\usepackage{marginfix}
\usepackage{marginnote}
\let\oldtocsection=\tocsection
\let\oldtocsubsection=\tocsubsection
\let\oldtocsubsubsection=\tocsubsubsection
\renewcommand{\tocsection}[2]{\hspace{0em}\textbf{\oldtocsection{#1}{#2}}}
\renewcommand{\tocsubsection}[2]{\hspace{1.8em}\oldtocsubsection{#1}{#2}}
\renewcommand{\tocsubsubsection}[2]{\hspace{3em}\oldtocsubsubsection{#1}{#2}}

\DeclareUnicodeCharacter{00A0}{ } 

\numberwithin{equation}{section}

\theoremstyle{plain}
\newtheorem{thm}{Theorem}[section]

\newtheorem{conj}[thm]{Conjecture}
\newtheorem{lemma}[thm]{Lemma}

\newtheorem{cor}[thm]{Corollary}

\theoremstyle{definition}

\newtheorem{definition}[thm]{Definition}

\theoremstyle{remark}
\newtheorem{remark}[thm]{Remark}

\newcommand{\R}{\mathbb{R}}

\newcommand{\C}{\mathbb{C}}

\newcommand{\N}{\mathbb{N}}

\renewcommand{\P}{\mathbb{P}}

\renewcommand{\H}{\mathbb{H}}

\newcommand{\cM}{\mathcal{M}}

\newcommand{\cS}{\mathcal{S}}
\newcommand{\cA}{\mathcal{A}}

\newcommand{\eps}{\varepsilon}

\renewcommand{\epsilon}{\varepsilon}
\renewcommand{\rho}{\varrho}
\renewcommand{\phi}{\varphi}

\newcommand{\1}{\mathrm{\mathbf{1}}}

\newcommand{\df}{\mathrm{d}}

\newcommand{\id}{\operatorname{id}}

\DeclareMathOperator{\Vol}{Vol}
\DeclareMathOperator{\InjRad}{InjRad}

 \def\Xint#1{\mathchoice
      {\XXint\displaystyle\textstyle{#1}}%
      {\XXint\textstyle\scriptstyle{#1}}%
      {\XXint\scriptstyle\scriptscriptstyle{#1}}%
      {\XXint\scriptscriptstyle\scriptscriptstyle{#1}}%
      \!\int}
   \def\XXint#1#2#3{{\setbox0=\hbox{$#1{#2#3}{\int}$}
        \vcenter{\hbox{$#2#3$}}\kern-.5\wd0}}
   
   \def\dashint{\Xint-}

\usepackage[usenames,dvipsnames]{xcolor}

\title[Short geodesics and $L^p$ norms on random surfaces]{Short geodesic loops and $L^p$ norms of eigenfunctions on large genus random surfaces}

\author{Clifford Gilmore} 
\address{School of Mathematical Sciences, University College Cork, Ireland}                                                              
\email{clifford.gilmore@ucc.ie}  

\author{Etienne Le~Masson} 
\address{Laboratoire de mathématiques AGM, UMR CNRS 8088, CY Cergy-Paris Universit\'e, France}
\email{etienne.le-masson@cyu.fr}

\author{Tuomas Sahlsten} 
\address{Department of Mathematics, University of Manchester, UK}
\email{tuomas.sahlsten@manchester.ac.uk}

\author{Joe Thomas}
\address{Department of Mathematics, University of Manchester, UK}
\email{joe.thomas@manchester.ac.uk}


 \thanks{C.~Gilmore was supported by the Magnus Ehrnrooth Foundation and the Irish Research Council via a Government of Ireland Postdoctoral Fellowship. E. Le Masson was supported by Initiative d'Excellence Paris//Seine. T. Sahlsten was supported by a start-up grant from MIMS, University of Manchester. J. Thomas was supported by the Dean's Award from the University of Manchester. }

\begin{document}
	\begin{abstract}
	
	We give upper bounds for $L^p$ norms of eigenfunctions of the Laplacian on compact hyperbolic surfaces in terms of a parameter depending on the growth rate of the number of short geodesic loops passing through a point. When the genus $g \to +\infty$, we show that random hyperbolic surfaces $X$ with respect to the Weil-Petersson volume have with high probability at most one such loop of length less than $c \log g$ for small enough $c > 0$. This allows us to deduce that the $L^p$ norms of $L^2$ normalised eigenfunctions on $X$ are $O(1/\sqrt{\log g})$ with high probability in the large genus limit for any $p > 2 + \eps$ for $\eps > 0$ depending on the spectral gap $\lambda_1(X)$ of $X$, with an implied constant depending on the eigenvalue and the injectivity radius.
	\end{abstract}

\maketitle

\section{Introduction}\label{intro}

\subsection{Background and main result} 

In the setting of a compact $n$-dimensional Riemannian manifold $(M,g)$, a deep understanding of the shape and asymptotics of eigenfunctions of the Laplacian is intimately linked to underlying geometric properties of the space itself. One means to realise this connection is through studying the $L^p$ norms of the eigenfunctions themselves. Indeed, as an example, primitive estimates show that the multiplicity of the eigenvalues are influenced by the $\sup$ norms of the eigenfunctions as well as the volume of the space through 
	\begin{align*}
		m(\lambda) \leq \Vol(M) \left(\max_{x\in M}\{|\psi(x)| : \Delta\psi = \lambda\psi,\ \|\psi\|_2=1\}\right)^2,
	\end{align*}
where $m(\lambda)$ is the multiplicity of the eigenvalue $\lambda$ (see for example the proof of Proposition 2.1 in \cite{Don01}).\par 

Eigenfunctions of the Laplacian feature prominently in quantum mechanics since they are precisely the states for which the probability measures $|\psi(x,t)|^2\,\df\Vol_M(x)$ are constants, where $\psi(x,t)$ is the free quantum evolution of a wavefunction $\psi(x)$. In this setting, a widely studied problem is to understand the properties of the eigenfunctions in the high energy, or large eigenvalue, limit, aiming to recover some characteristics of the classical dynamics, for example in the study of Quantum (Unique) Ergodicity \cite{Sni74, Zel87, DeV85, RS94, Lin06, Has10}. 

In the large eigenvalue aspect, Sogge's~\cite{Sog88} seminal work identified the link between the growth of $L^p$ norms of eigenfunctions and their $L^2$ norms in terms of their eigenvalue. In particular, if $\Delta\psi = \lambda\psi$ then
	\begin{align*}
		\|\psi\|_p \lesssim_M \lambda^{\sigma(n,p)}\|\psi\|_2,
	\end{align*}
where
	\begin{align*}
		\sigma(n,p)=\begin{cases}
			\frac{n}{2}\left(\frac{1}{2}-\frac{1}{p}\right)-\frac{1}{4}, & \text{if } \frac{2(n+1)}{n-1} \leq p \leq \infty,\\
			\frac{n-1}{4}\left(\frac{1}{2}-\frac{1}{p}\right), & \text{if } 2 \leq p \leq \frac{2(n+1)}{n-1}.
		\end{cases}
	\end{align*}
Here we use  $\lesssim$ to denote that the quantity is bounded up to a constant and subscripts are used to indicate specific dependencies of such a constant. These bounds are sharp on the sphere and attained for high $p$ by zonal spherical harmonics (concentration of the mass around a point) and for low $p$ by Gaussian beams (concentration along closed geodesics). However, in the case of manifolds of non-positive curvature (or without conjugate points), B\'{e}rard \cite{Ber77} had previously obtained a logarithmic improvement of the sup norm. This was more recently extended to values of $p>\frac{2(n+1)}{n-1}$ by Hassell and Tacy \cite{HT15}. \par 

The implied constant in the Sogge bound was investigated by Donnelly \cite{Don01} to reveal that the underlying geometry again plays an important role. More specifically, the constant depends upon bounds on the injectivity radius and the sectional curvature of the manifold. 

In this paper we restrict our attention to hyperbolic surfaces and investigate the influence of the geometry on $L^p$ norms.  Rather than seeking bounds in terms of eigenvalues, we focus on their dependence on the growth rate of short geodesic loops (see \eqref{eqn: ShortGeodesicBound} below). Our goal is to understand this geometric connection with random hyperbolic surfaces, using integration tools on the moduli space developed by Mirzakhani \cite{Mir07, Mir07JAMS, Mir13, MZ15}. In \cite{Mir13}, Mirzakhani initiated a theory of large genus random surfaces with respect to the Weil-Petersson volume (see Section \ref{s:randombackground} for background). An important success of these methods was the proof by Mirzakhani and Petri \cite{MP17} that the length of short geodesics on random surfaces follow a Poisson distribution in the large genus limit. From there, it is natural to try to connect the behaviour of closed geodesics to the spectrum of random surfaces via Selberg's theory (see for example \cite{Ber16} for background on Selberg's trace formula). We present in this paper one of the first attempts at such a connection between the geometry of random surfaces and eigenfunctions of the Laplacian.

A central motivation in our work is to understand the delocalisation properties of eigenfunctions on large volume manifolds. In a recent article \cite{LMS17}, Le~Masson and Sahlsten proposed a version of quantum ergodicity for hyperbolic surfaces of large genus.\footnote{Note that by the Gauss-Bonnet theorem the genus $g$ and the volume of a compact hyperbolic surface $|X|$ are related by the formula $|X| = 2\pi (2g-2),$ and are therefore equivalent parameters in this context.} The theorem is a delocalisation result analogous to the quantum ergodicity theorem of \v{S}nirel'man~\cite{Sni74}, Zelditch~\cite{Zel87} and Colin de Verdi\`ere~\cite{DeV85}, but valid in the large volume limit, and for eigenfunctions in a bounded spectral interval. 
This result was inspired by corresponding theorems on regular graphs \cite{ALEM15, BLML15}, viewed as discrete analogues of hyperbolic surfaces. We will follow a similar heuristic to push the graph methods of Brooks and Le~Masson \cite{BLM18} to the continuous setting. This deterministic aspect will be combined with new estimates on short geodesics of random surfaces, to obtain bounds on $L^p$ norms for random surfaces. Recently, there have been major breakthroughs in the study of eigenvectors on random regular graphs, from optimal sup-norm bounds \cite{BHY19} to the proof of their Gaussian behaviour \cite{BS19}. Our hope is to have provided a stepping stone towards the adaptation of these more advanced results. 

\par 
Before we state our main theorem, let us define the model of random surfaces we are considering. For any $g \geq 2$, we denote by $\mathcal{M}_g$ the moduli space of compact hyperbolic surfaces of genus $g$. It can be seen as a quotient $\mathcal{M}_g = \mathcal{T}_g / \text{Mod}_g$ of the Teichm\"uller space $\mathcal{T}_g$ by the mapping class group $\text{Mod}_g$ (see Section \ref{s:randombackground} for definitions). The Teichm\"uller space $\mathcal{T}_g$ is equipped with a symplectic form $\omega_g$ called the Weil-Petersson form that is invariant under the action of $\text{Mod}_g$. The associated volume form then descends to the quotient $\mathcal{M}_g$, which is of finite total volume. Denoting by $\text{Vol}_{wp}(A)$ the Weil-Petersson volume of a measurable set $A \subset \cM_g$, we obtain the probability measure 
$$ \mathbb{P}_g (A) = \frac{\text{Vol}_{wp}(A)}{\text{Vol}_{wp}(\cM_g)}.$$
One of the remarkable achievements of Mirzakhani was to compute the Weil-Petersson volume of $\cM_g$, and more generally of the moduli spaces of surfaces with boundaries and punctures, making it possible to estimate such probabilities. Note that an alternative model of compact random surfaces has been developed by Brooks and Makover \cite{BM04}. This model is not equivalent in general and we will only work here with the Weil-Petersson model.

For a compact hyperbolic surface $X$, we denote by $\InjRad(X)$ its injectivity radius, which is half the length of its shortest geodesic loop. The main theorem we prove is the following. 

\begin{thm}\label{thm:MainRandom}
Let $X$ be a random compact hyperbolic surface of genus $g$ distributed according to $\P_g$. There exists a universal constant $\delta > 0$ such that for any $c > 0$ and $0 < b < \frac12$, we have the following bounds with probability $1- O\left(g^{-\frac12 + \delta (c+b)} + g^{-2b}\right)$. For an eigenfunction $\psi_\lambda$ of the Laplacian with eigenvalue $\lambda\geq\frac{1}{4}$, 
	\begin{align*}
		\|\psi_\lambda\|_p \lesssim_{p,\lambda,c} \frac{1}{\sqrt{\mathrm{InjRad}(X)\log(g)}}\|\psi_\lambda\|_2,
	\end{align*}
for any $2+4\beta < p \leq\infty$, where $\beta\in [0,\frac{1}{2})$ is such that the smallest non-zero eigenvalue of the Laplacian on $X$ is at least $\frac{1}{4}-\beta^2$. Moreover, if $\psi_\lambda$ is an eigenfunction of the Laplacian with eigenvalue $\lambda\in [0,\frac{1}{4}-\varepsilon)$ for some $\varepsilon>0$ then
	\begin{align*}
		\|\psi_\lambda\|_p \lesssim \frac{1}{\mathrm{InjRad}(X)\left(g^{c\sqrt{\varepsilon}}-1\right)^{1-\frac{2}{p}}}\|\psi_\lambda\|_2,
	\end{align*}
for any $2<p\leq\infty$.
\end{thm}

\begin{remark}
\label{remark: injrad to logg}
\begin{itemize}
\item[(1)] The parameters $c$ and $b$ in Theorem \ref{thm:MainRandom} can be chosen sufficiently small such that the probability tends to $1$ as $g\to+\infty$. We interpret this result as saying that for a random surface of large fixed genus, the given bounds are true \emph{with high probability}. The implied constant in the first inequality is continuous in the eigenvalue and independent of the genus. Therefore if we fix any bounded interval $I \subset (1/4, +\infty)$, we have a decay of $\|\psi_\lambda\|_p/\|\psi_\lambda\|_2$ for any eigenvalue $\lambda \in I$ when $g \to +\infty$. However we emphasise that this is not a result on sequences of independent random surfaces in the sense of the product probability $\prod_k \P_{g_k} $ with $(g_k)_{k\in\N}$ such that $g_k \to +\infty$ when $k \to +\infty$.
\item[(2)] Note the different behaviour between the two parts of the spectrum: $[1/4, +\infty)$, to which we will often refer as the \emph{tempered spectrum}, and $(0,1/4)$ called the \emph{untempered spectrum}.
\item[(3)]By Theorem 4.2 of \cite{Mir13}, we have that for any $\alpha>0$
	\begin{align*}
		\mathbb{P}_g(X:\mathrm{InjRad}(X)\geq \log(g)^{-\alpha}) \geq 1 - O(\log(g)^{-2\alpha}).
	\end{align*}
This means one can remove the injectivity radius constants in the above result and obtain in the tempered case,
	\begin{align*}
		\|\psi_\lambda\|_p \lesssim_{p,\lambda,c} \frac{1}{\log(g)^{\frac{1}{2}(1-\alpha)}}\|\psi_\lambda\|_2,
	\end{align*}
and in the untempered case,
	\begin{align*}
		\|\psi_\lambda\|_p \lesssim \frac{\log(g)^\alpha}{\left(g^{c\sqrt{\varepsilon}}-1\right)^{1-\frac{2}{p}}}\|\psi_\lambda\|_2,
	\end{align*}
for any $\alpha > 0$ and $c>0$ as before both occurring with probability tending to $1$ as $g\to +\infty$ by a union bound.

\end{itemize}
\end{remark}

The probability bound in Theorem \ref{thm:MainRandom} is governed by the measure of a certain subset of $\cM_g$, within which surfaces satisfy the norm inequalities of the theorem. This subset which we denote by $\cA_g^{b,c}$ depends upon two parameters $b,c>0$ chosen independently of the genus and the construction of this set is the subject of the next result we describe. On its own, this result is a statement about short geodesic loops on random surfaces existing in the $g^{-b}$-thick part of the moduli space for some $b>0$ to be chosen (independently of $g$), and so we isolate it as it can in itself be of interest.\par 
By the $a$-thick part of the moduli space, we mean the collection of $X\in\cM_g$ whose injectivity radius is at least $a$; this space is often denoted by $(\cM_g)_{\geq a}$. Again by Theorem 4.2 of \cite{Mir13}, we have that
	\begin{align*}
		\P_g((\cM_g)_{\geq g^{-b}}) \geq 1-O(g^{-2b}).
	\end{align*}
This accounts for the origin of the parameter $b$ in the set $\cA_g^{b,c}$. The parameter $c$ originates from a condition on the number of primitive geodesic loops of length at most $c\log(g)$ (for suitably chosen $c>0$ independent of the genus $g$) that can be based at any point on the surface. For this we introduce the following random variable. For any $X \in \cM_g$, let us denote by $N_L(X,x)$ the number of primitive geodesic loops $\gamma$ (not necessarily simple) of length $\ell_X(\gamma) \leq L$ based at a point $x \in X$, and set 
$$ N_L(X) = \sup_{x\in X} N_L(X,x). $$
The set $\cA_g^{b,c}$ is then the collection of surfaces in the $g^{-b}$-thick part of the moduli space that also have $N_{c\log(g)}(X)\leq 1$ for appropriately chosen constants $b$ and $c$ that are implicit and described in more detail in Section \ref{s:randomproof}. In other words,
$$\cA_g^{b,c} = \left\{ X \in (\cM_g)_{\geq g^{-b}} : N_{c\log(g)}(X)\leq 1 \right\}.$$
We have the following result about the probability for a random surface to be in this set.
\begin{thm}\label{t:proba}
There exists $\delta > 0$ such that for all $c > 0$ and $0<b<\frac{1}{2}$,
$$\P_g\left(X\in(\mathcal{M}_g)_{\geq g^{-b}} : N_{c\log g}(X) \leq 1 \right) \geq 1- O\left(g^{-\frac12 + \delta (c+b)}+g^{-2b}\right)\ \ \ \text{as}\ g\to\infty,$$
and therefore for $b$ and $c$ small enough, this probability tends to $0$ when $g \to +\infty$.
\end{thm}

\begin{remark}
The fourth named author has ongoing work with Laura Monk in which new methods can be used to obtain more quantitative bounds on the probability of $N_{c\log(g)}\leq 1$.
\end{remark}

Thus, the rate at which the probability holds in Theorem \ref{thm:MainRandom} is given by the rate in Theorem \ref{t:proba}. The previous theorem says that, with high probability when $g \to +\infty$, at any point of a random surface in the $g^{-b}$-thick part of the moduli space there is no more than one primitive geodesic loop of length less than $c \log g$ based at this point. This implies in particular that if there is one, this loop is necessarily simple since the shortest geodesic loop based at a point is always simple. Related countings of the number of short geodesics of a given length on random surfaces are done in \cite{MP17}. We use similar ideas to this work, but the dependence of the length of the loops we consider on the genus (as opposed to just being uniformly bounded) and the fact that we consider geodesic loops rather than just closed geodesics requires us to develop more delicate and quantitative tools that we detail in Section \ref{s:randomproof}.

Theorem \ref{thm:MainRandom} relies on Theorem \ref{t:proba} together with a deterministic theorem about $L^p$ norms. This deterministic result requires us to consider a condition on the surfaces and it is precisely Theorem \ref{t:proba} that allows us to dispense of this condition in favour of a result holding with high probability. Let $X=\Gamma\backslash\H$ be a compact hyperbolic surface with fundamental domain $D$. We fix $R = R(X) \geq 0$ and $C(X) > 0$ such that for any $\delta >0$ there exists $C_0(\delta) > 0$, so that
	\begin{align} \label{eqn: ShortGeodesicBound}
		\sup_{z,w \in D} | \{ \gamma \in \Gamma \, | \, d(z,\gamma w) \leq r \} | \leq C(X) C_0(\delta) \, e^{\delta r} \qquad \text{for any } r \leq R.
	\end{align}

As explicated in Lemma \ref{lem: Sufficient Assumption}, if there exist $n \in \N$ and $L > 0$ such that $N_L(X) \leq n$, then we can take $C(X) = \frac1{\InjRad(X)}$ and $R = L$ in \eqref{eqn: ShortGeodesicBound}. The idea here is to take $R$ as large as possible while controlling the constant $C(X)$. On random surfaces Theorem \ref{t:proba} allows us to take $R=c\log(g)$ with $C(X)=\frac{1}{\InjRad(X)}$ and the constant $C(X)$ can be dispensed of by Remark \ref{remark: injrad to logg} if so desired.\par 

Our deterministic result in terms of the parameters of condition \eqref{eqn: ShortGeodesicBound} is stated as follows.

\begin{thm}  \label{thm:Main}
Suppose that $X=\Gamma\backslash\H$ is a compact hyperbolic surface whose smallest non-zero eigenvalue of the Laplacian is at least $\frac{1}{4}-\beta^2$ for some $\beta\in [0,\frac{1}{2})$. For an eigenfunction $\psi_\lambda$ of the Laplacian with eigenvalue $\lambda\geq\frac{1}{4}$, we have that
	\begin{align*}
		\|\psi_\lambda\|_p \lesssim_{p,\lambda} \frac{\sqrt{C(X)}}{\sqrt{R}}\|\psi_\lambda\|_2,
	\end{align*}
for any $2+4\beta < p \leq\infty$. Moreover, if $\psi_\lambda$ is an eigenfunction of the Laplacian with eigenvalue $\lambda\in [0,\frac{1}{4}-\varepsilon)$ for some $\varepsilon>0$ then for any $\delta > 0$
	\begin{align*}
		\|\psi_\lambda\|_p \lesssim_\delta \frac{C(X)}{\left(e^{(1-\delta)\sqrt{\varepsilon}R}-1\right)^{1-\frac{2}{p}}}\|\psi_\lambda\|_2,
	\end{align*}
for $2<p\leq\infty$. Throughout, $C(X)$ and $R$ are given by \eqref{eqn: ShortGeodesicBound}.
\end{thm}

Note that in a purely deterministic setting we can always take $R = \InjRad(X)$ and $C(X)=1$. On the other hand, by noticing that the number of fundamental domains intersecting a ball of radius $R$ is greater than $e^R/ g$ we see that if $C(X)$ is taken to be independent of $X$ we have necessarily $R \leq \log(g)$.

\bigskip
\subsubsection*{Optimal bounds}
One can ask what is the best bound on $L^p$ norms that can be obtained in the large genus limit. Clearly for any function $\psi \colon X \to \R$
\begin{equation}\label{e:bestbound}
\|\psi\|_\infty \geq \frac{\|\psi\|_2}{\sqrt{|X|}}
\end{equation}
with equality if and only if $\psi$ is constant almost everywhere. Eigenfunctions of non-zero eigenvalue are not constant and therefore some correction is required. 

On large random regular graphs, the following was proved by Bauerschmidt, Huang and Yau \cite[Theorem 1.2]{BHY19}. Let $\mathrm{G}_{N,d}$ be the set of random regular graphs of degree $d$ on $N$ vertices. We put the uniform probability measure on $\mathrm{G}_{N,d}$. There exists $d_0$ very large but fixed such that for $d \geq d_0$ and with probability tending to $1$ when $N \to +\infty$, any eigenvector $v$ with eigenvalue in the tempered spectrum satisfies 
$$ \| v \|_\infty \lesssim \frac{(\log N)^\alpha}{\sqrt{N}} \|v\|_2, $$
for some $\alpha > 0$ depending on the distance of the eigenvalue from the boundaries of the tempered spectrum.

Inspired by this graph result we can formulate the following conjecture.

\begin{conj} \label{c:QUE}
Let $X$ be a compact hyperbolic surface of genus $g$ chosen uniformly at random with respect to the Weil-Petersson volume. Then for any $\eps >0$ and any eigenfunction $\psi_\lambda$ with eigenvalue $\lambda \in \left(\frac14 + \eps, +\infty \right)$ we have

$$ \| \psi_\lambda \|_\infty \lesssim \frac{(\log g)^{\alpha(\eps)}}{\sqrt{g}} \| \psi_\lambda \|_2 $$
for some function $\alpha(\eps) > 0$ of $\eps$, with probability tending to $1$ when $g \to +\infty$.
\end{conj}

Such a result on the sup norm would give a strong form of delocalisation. In particular it would prevent concentration of eigenfunctions on sets of volume less than $g/\log(g)^{2\alpha}$.

\subsubsection*{Arithmetic surfaces}
 In the compact arithmetic setting, and for a Hecke eigenfunction $\psi_\lambda$, stronger bounds exist both in terms of the eigenvalue, due to Iwaniec and Sarnak \cite{IS95}, and in terms of the genus (or more precisely the congruence level), due to Saha and Hue-Saha \cite{Sah19, HS19}. In the eigenvalue aspect, for a given compact arithmetic surface the following bound holds \cite[Theorem 0.1]{IS95}: for any eigenfunction $\psi$ with $\Delta \psi = \lambda \psi$ and any $\eps >0$
 $$ \| \psi \|_\infty \lesssim_{\eps} \lambda^{5/24 + \eps}  \| \psi \|_2.$$
 
In the level aspect the bound is more complex and depends on the arithmetic properties of the level but it has a power decay in terms of the genus of the form
 $$ \| \psi \|_\infty \lesssim_{\lambda} g^{-\alpha}  \| \psi \|_2$$
for some exponent $\alpha > 0$.
 Note that similar level aspect bounds have been obtained previously in the non-compact case of congruence covers of the modular surface \cite{BH10, HT13}.

\subsubsection*{Hybrid bounds} 
The bounds we obtain depend implicitly on the eigenvalue. The dependence can be made explicit in our proof but is much worse than Sogge's bounds. It would be interesting to have better combined dependence both in terms of eigenvalue and genus. Such hybrid bounds were obtained in the arithmetic setting for Maass cusp forms by Templier and Saha \cite{Tem15, Sah17}. Developing such a theory on random surfaces could for example allow one to improve eigenvalue bounds for a positive measure set of surfaces. Alternatively, in a similar way as the work of Bauerschmidt, Huang and Yau \cite{BHY19} requires graphs of very large degree, we could expect that Conjecture \ref{c:QUE} could be easier to approach if we assume the eigenvalue $\lambda$ to be large.

\subsubsection*{Multiplicities}
As we have observed at the beginning of the introduction, the sup norm of an $L^2$-normalised eigenfunction $\psi_\lambda$ with eigenvalue $\lambda$ can be linked to the multiplicity of $\lambda$ by
\begin{equation}\label{e:mult}
\frac{m(\lambda)}g \lesssim \|\psi_\lambda\|_\infty^2.
\end{equation}
Through this inequality our result is connected to the problem of limit multiplicities in representation theory initiated by DeGeorge and Wallach \cite{DEGW78,DEGW79}. Bounds for multiplicities in arithmetic settings have been studied by Sarnak and Xue \cite{SX91}. Recently \cite{ABBGNRS17}, it was proved that for a general Benjamini-Schramm converging sequence of compact hyperbolic surfaces $(X_n)$ with associated genus $g_n \to +\infty$, for any $\lambda >0$ the ratio $m(\lambda)/g \to 0$ when $g_n \to +\infty$. Note that a sequence of random compact hyperbolic surfaces of increasing genus converges in the sense of Benjamini-Schramm to the hyperbolic plane with high probability (\cite[Section 4.4]{Mir13}). In this case our theorem provides a rate via \eqref{e:mult} and Remark \ref{remark: injrad to logg}(3).

\begin{cor}
Let $X$ be a random compact hyperbolic surface of genus $g$ chosen according to the probability $\mathbb{P}_g$. Denote by $m(\lambda)$ the multiplicity of an eigenvalue $\lambda \in (0, +\infty)$. Then there exists a universal constant $d>0$ such that for any $\alpha >0$ the following bounds occur with probability $1- O\left((\log g)^{-2\alpha}\right)$.
$$ \frac{m(\lambda)}g \lesssim_{d,\lambda} \frac1{(\log g)^{1-\alpha}},$$
for tempered eigenvalues $\lambda\in (\frac14, +\infty)$ and,
$$ \frac{m(\lambda)}g \lesssim_d \frac{(\log g)^{2\alpha}}{g^{2d \sqrt{\eps}}},$$
for untempered eigenvalues $\lambda\in (0, \frac14 -\eps)$. 
\end{cor}

It is also possible here to take $\alpha = 0$ if we add a factor $\frac1{\InjRad(X)}$ in the tempered spectrum bound and $\frac1{\InjRad(X)^2}$ for the untempered spectrum by directly using Theorem \ref{thm:MainRandom}. This would also modify the probability to that given in Theorem \ref{thm:MainRandom}.

The constant $d > 0$ here originates from the constant c in the length $c \log(g)$ of closed geodesic loops that we can control (see Theorem \ref{t:proba}). In our case $c$, and hence $d$, can be very small and is not explicit. To make it explicit and optimise it, we would need a more careful analysis of the product in Lemma \ref{lem: Sum-Volume Product}, which in turns requires more precise estimates than the ones in \cite{MZ15}. 

\subsubsection*{Optimal spectral gap}
In the case of untempered eigenvalues, we expect that for any $\epsilon > 0$, and $\lambda \in (0, \frac14 - \epsilon)$, the multiplicity of $m(\lambda)$ tends to $0$ when $g\to +\infty$, implying that the spectral gap is close to being optimal with high probability in the large genus limit (see \cite[Section 10.4]{Wri19}). This can be seen as a random surfaces analogue of Selberg's $\frac14$ conjecture. It is likely that a more quantitative understanding of Theorem \ref{t:proba} --- and therefore a more explicit constant $c$ --- is required to prove such a result on random surfaces (see for example how such properties on short loops are used to prove an analogous theorem on regular graphs \cite{Bor15}). However, improving the sup norm bound for untempered eigenfunctions can only give at best $m(\lambda) \leq 1$ by an inequality such as \eqref{e:mult}, due to the absolute lower bound on sup norms \eqref{e:bestbound}. On the other hand, an optimal spectral gap theorem for random surfaces would improve Theorem \ref{thm:MainRandom} by extending the validity of the bound down to $p > 2$. 
To our knowledge, the only known lower bound for the spectral gap currently in the Weil-Petersson model is due to Mirzakhani \cite{Mir13} who proved that the first non zero eigenvalue is greater than $(\frac{\ln 2}{\pi + \ln 2})^2/4$ with probability tending to $1$ when $g \to +\infty$. In the Brooks-Makover model, a non-explicit uniform lower bound on the spectral gap has also been shown to exist with high probability \cite{BM04}.
More recently in a model based on random coverings, strong explicit spectral gaps have been obtained by Magee and Naud \cite{MN} for non-compact surfaces.

\subsection{Outline of the article} Aside from the introduction, this article consists of five other sections organised as follows.

\begin{enumerate}
	\item Section 2: An overview of the preliminaries of the harmonic analysis used in the proof of the deterministic results.
	\item Section 3: The proof of Theorem \ref{thm:Main} in the case of a hyperbolic surface with optimal spectral gap.
	\item Section 4: The proof of Theorem \ref{thm:Main} in the case of a hyperbolic surface with an arbitrary spectral gap.
	\item Section 5: An overview of the preliminaries of the Teichm\"{u}ller and random surface theory used in the proof of the probabilistic results.
	\item Section 6: The proofs of Theorem \ref{t:proba} and Theorem \ref{thm:MainRandom}.
\end{enumerate}

\section{Harmonic Analysis on Hyperbolic Surfaces}
\label{sec: Harmonic Analysis Background}

In this section, we introduce some background necessary for our investigation. Much of what is found here is standard and we refer to Katok \cite{Kat92} for the background on hyperbolic geometry and Bergeron \cite{Ber16} and Iwaniec \cite{Iwa02} for the background on invariant integral operators and the Selberg transform.\par 
We will work with the Poincar\'{e} upper half-plane as a model for the hyperbolic plane
\begin{equation*}
\H = \left\lbrace z=x + iy \in \C : y >0 \right\rbrace
\end{equation*}
which is equipped with the standard hyperbolic Riemannian metric 
\begin{equation*}
ds^2 = \frac{dx^2 + dy^2}{y^2}.
\end{equation*}
The distance between two points $z, z' \in \H$ with respect to the metric is denoted by $d(z, z')$ and the associated hyperbolic volume is given by
\begin{equation*}
d\mu(z) = \frac{dx \, dy}{y^2}.
\end{equation*}

We identify the group of orientation-preserving isometries of $\H$ with the projective special linear group $\mathrm{PSL}(2, \R)$, which contains the $2 \times 2$ matrices, with real entries, that have determinant 1 modulo $\pm I_2$, where $I_2$ the $2\times 2$ identity matrix.  The group acts transitively on points $z \in \H$ via M\"{o}bius transformations
\begin{equation*}
 \begin{pmatrix}
 a & b \\
 c & d
 \end{pmatrix} \cdot z 
= 
\frac{az + b}{cz + d} ,
\end{equation*}
where $\displaystyle \begin{pmatrix}
a & b \\
c & d
\end{pmatrix} \in \mathrm{PSL}(2, \R).$

A \emph{hyperbolic surface} can be seen as a quotient $X = \Gamma \backslash \H$, where $\Gamma\leq\mathrm{PSL}(2,\R)$ is a fixed point free Fuchsian group. In other words, $\Gamma$ is a fixed point free discrete subgroup of $\mathrm{PSL}(2, \R)$.  Denote by $D\subseteq\H$ a fundamental domain associated with $\Gamma$. The Riemannian metric on $\H$ is then naturally inherited by the quotient in the standard way as a Riemannian manifold quotient since the group acts isometrically.

The \emph{injectivity radius} on the surface $X = \Gamma \backslash \H$ at a point $z$ is defined as
\begin{equation*}
\InjRad_X(z) = \frac12 \inf \left\{ d(z, \, \gamma z) : \gamma \in \Gamma \setminus \{\pm \id \} \right\}
\end{equation*}
and this gives the largest $R>0$  such that the ball $B_X (z, R)$ is isometric to a ball of radius $R$ in the hyperbolic plane.  In the case when the surface $X$ is compact, there exists a universal positive lower bound for the injectivity radius at each of the points. This allows for the injectivity radius of a compact surface $X$ to be defined as 
\begin{equation*}
\InjRad(X) = \inf_{z \in X}  \InjRad_X (z)>0.
\end{equation*}

We say that a bounded measurable kernel $K \colon \H \times \H \to \C$ is invariant under the diagonal action of $\Gamma$ if for any $\gamma \in \Gamma$ we have
\begin{equation*}
K(\gamma \cdot z, \gamma \cdot w) = K(z, w), \quad (z, w) \in \H \times \H .
\end{equation*}
Such kernels are also referred to as \textit{point-pair invariants}.

A \emph{radial kernel} $k \colon [0, +\infty] \to \C$ is a bounded, measurable, function. Given such a kernel, the mapping $K \colon \H\times\H\to\C$ given by
\begin{equation*}
(z, w) \mapsto k(d(z, w))
\end{equation*}
is an invariant kernel for $(z, w) \in \H \times \H$. Conversely, an invariant kernel gives rise to a radial kernel in an obvious way and so the two can be identified. \par 
To construct an invariant integral operator on the surface $\Gamma\backslash \H$, we firstly note that functions on $X$ are naturally identified with $\Gamma$-periodic functions on a fundamental domain $D\subseteq\H$. Given an invariant kernel $K$, we then define an associated automorphic kernel on $D\times D$ by
	\begin{align*}
		K_\Gamma (z,w)=\sum_{\gamma\in \Gamma} K(z,\gamma w).
	\end{align*}
This summation converges if one imposes an appropriate decay condition on the kernel $k$, such as the existence of some $\delta>0$ such that
	\begin{align*}
		|k(\rho)|=O\left(e^{-(1+\delta)\rho}\right).
	\end{align*}
With this, we may define an associated invariant integral operator $A$ on the surface $X$ by
\begin{equation*}
Af(z) =  \int_D \sum_{\gamma \in \Gamma} K(z, \gamma w) f(w) \: \df\mu(w) 
\end{equation*}
for any $\Gamma$-invariant function $f$ and $z \in D$.

The importance of the radial operators is derived from their connection to the Laplacian.  The Laplacian $\Delta$ on $\H$ is given in coordinates $z=x+iy$ by the differential operator
\begin{equation*}
\Delta = -y^2 \left( \frac{\partial^2}{\partial x^2} + \frac{\partial^2}{\partial y^2} \right),
\end{equation*}
and since the Laplacian commutes with isometries it can be considered as a differential operator on any hyperbolic surface $\Gamma \backslash \H$.\par 
In the case that $X=\Gamma\backslash\H$ is a compact surface, the spectrum of the Laplacian on $X$ denoted by $\sigma_X(\Delta)$ is discrete and contained in the interval $[0,\infty)$. Moreover, the eigenfunctions corresponding to the eigenvalue 0 are all constant functions and thus in particular, the corresponding eigenspace is one dimensional. From the general theory of the Laplacian on compact Riemannian manifolds, there exists a sequence $0=\lambda_0\leq \lambda_1\leq \ldots\to\infty$ and an orthonormal basis $\{\psi_{\lambda_i}\}_{i\geq 0}$ of $L^2(\Gamma\backslash\H)\cong L^2(D)$  such that
	\begin{align*}
		\Delta\psi_{\lambda_i}=\lambda_i\psi_{\lambda_i},
	\end{align*}
that is, $\psi_{\lambda_i}$ is an eigenfunction corresponding to the eigenvalue $\lambda_i$. In the case of a hyperbolic surface, it is instructive to partition the spectrum into two parts: the \textit{tempered spectrum} which corresponds to the portion of the spectrum inside $[\frac{1}{4},\infty)$ and the \textit{untempered spectrum} corresponding to $[0,\frac{1}{4})$. When a surface has $\sigma_X(\Delta)\subseteq \{0\}\cup [\frac{1}{4},\infty)$, we say that it has \textit{optimal spectral gap}.\par

We recall that any eigenfunction of the Laplacian is also an eigenfunction of an invariant integral operator. The corresponding eigenvalue of the integral operator is determined by the \emph{Selberg transform} $\mathcal{S}(k)$ of the radial kernel $k \colon [0, \infty] \to \C$, which is defined as the Fourier transform 
\begin{equation*}
	\mathcal{S} (k)(r) = h(r) = \int_{-\infty}^{+\infty} e^{iru} g(u) \: \df u
\end{equation*}
of the function
\begin{equation*}
g(u) = \sqrt{2} \int_{|u|}^{+\infty}   \frac{k(\rho) \sinh \rho}{\sqrt{\cosh \rho - \cosh u}} \: \df\rho.
\end{equation*}
Conversely, given a suitable function $h$, one can construct a kernel $k$ via taking an inverse Selberg transform of $h$ such that the associated automorphic kernel $K_\Gamma$ defined above converges. More specifically, if $h:\{z\in\mathbb{C}: |\mathrm{Im}(z)|\leq \frac{1}{2}+\varepsilon\}\to\mathbb{C}$ for some $\varepsilon>0$ satisfies 
	\begin{enumerate}
		\item $h$ is analytic,
		\item $h$ is even,
		\item $h$ satisfies a decay condition of the form 
			\begin{align*}
				|h(z)| = O\left(1+|z|^2)^{-1-\varepsilon}\right),
			\end{align*}
	\end{enumerate}
then the inverse Selberg transform of $h$ is defined through 
$$ g(u) = \frac1{2\pi}  \int_{-\infty}^{+\infty} e^{-isu} h(s) \, \df s $$
 and then
$$k(\rho) = -\frac1{\sqrt{2}\pi} \int_\rho^{+\infty} \frac{g'(u)}{\sqrt{\cosh u - \cosh \rho}} \, \df u.$$
We note that under these conditions (see for example \cite{Mar12}), $k$ satisfies the decay condition mentioned previously with $\delta=\varepsilon$. We will utilise these decay conditions on such a function $h$ within this paper. With $h$ satisfying these conditions, we obtain the following.

\begin{thm}[{\cite[Sections 3.3, 3.4]{Ber16}} or {\cite[Theorem 1.14]{Iwa02}}]
\label{thm: BergeronThmOnSelberg}
Let $X=\Gamma\backslash\H$ be a hyperbolic surface and $k \colon [0,\infty]\to\C$ a radial kernel. Suppose that $\psi_\lambda$ is an eigenfunction of the Laplacian with eigenvalue $\lambda=s_\lambda^2+\frac{1}{4}$ for $s_\lambda\in\C$.  Then $\psi_\lambda$ is an eigenfunction of the convolution operator $A$ with invariant kernel $k$ and
	\begin{align*}
		(A\psi_\lambda)(z) = \int_\H k(d(z,w)) \psi_\lambda(w) \: \df\mu(w)=h(s_\lambda)\psi_\lambda(z),
	\end{align*}
where $h(s_\lambda)=\mathcal{S}(k)(s_\lambda)$.
\end{thm}

\section{Deterministic Bounds for Surfaces with Optimal Spectral Gap}
\label{sec: MaxSpecGap}
Consider a compact hyperbolic surface $X=\Gamma\backslash\H$ with $D\subseteq\H$ a fundamental domain of $X$ and assume the short closed geodesics condition \eqref{eqn: ShortGeodesicBound} holds for some $R\geq 0$ on $X$ with surface constant $C(X)$. In this section, we additionally assume that $X$ has an optimal spectral gap so that $\sigma_X(\Delta)\subseteq \{0\}\cup [\frac{1}{4},\infty)$. In this case, by letting $\lambda=s_\lambda^2+\frac{1}{4}$ be the parametrisation of the eigenvalue $\lambda$ of the Laplacian as described in Section \ref{sec: Harmonic Analysis Background}, then $s_\lambda$ is either in $[0,\infty)$ or is equal to $\frac{1}{2}i$, with the latter case occurring when $\lambda=0$. 

The extra assumption on the surfaces here provides for slightly stronger results as emphasised in Theorem \ref{thm:Main}. Moreover, the crux of the methodology that we use to prove the result can be demonstrated without the additional technicalities that are brought about by the small eigenvalues. In fact, for this reason we also defer the proof of the result for untempered eigenfunctions to Section \ref{sec: ArbitrarySpecGap} and focus solely on the tempered portion of the spectrum. 

\subsection{Outline of the proof}
The $L^p$ norm bounds for tempered eigenfunctions in Theorem \ref{thm:Main} is proven through the following methodology. We will use Selberg's theory to build a convolution operator $W_{R,\lambda}$ satisfying on the spectral side
$$\|W_{R,\lambda} \psi_\lambda\|_p \gtrsim_\lambda R \|\psi_\lambda\|_p$$
for any eigenfunction $\psi_\lambda$ of eigenvalue $\lambda \geq \frac14$, and on the geometric side
$$ \|W_{R,\lambda} \|_{L^2(X) \to L^p(X)} \lesssim_p \sqrt{R},$$
with $R$ being given by \eqref{eqn: ShortGeodesicBound}.
The latter inequality will be obtained via a $TT^*$ argument:
$$\|T\|_{L^2(X)\to L^p(X)}^2=\|TT^*\|_{L^q(X)\to L^p(X)}.$$
For this purpose:
\begin{enumerate}
	\item We firstly define via the inverse Selberg transform a family of operators $P_t$ that can be seen as a smoothened version of the wave cosine kernel $\cos (t\sqrt{\Delta})$ and which will be used as a building block for our operator $W_{R,\lambda}$.
	\item Preparing for the $TT^*$ argument, we next prove a linearisation formula of the type $$P_t P_s^* = \frac12\left(Q_{t+s} + Q_{|t-s|}\right),$$ 
	where $Q_t$ is a family of operators studied previously by Brooks and Lindenstrauss \cite{BL14}. This is done looking at the spectral action of the operators via the Selberg transform. (Lemma \ref{lem: LinearisationResult}) 
	\item We use relevant bounds obtained in \cite{BL14} (reproduced in Lemma \ref{lem: BrooksLindenstraussBounds}) to bound the operator norms of $Q_t$ for $t \leq R$, where $R$ is given by \eqref{eqn: ShortGeodesicBound}. (Lemma \ref{lem: QtBounds})
	\item The operator $W_{R,\lambda}$ is then defined roughly as
	$$W_{R,\lambda}=\int_0^R \cos(s_\lambda t)P_t \, \df t.$$
	\item We realise the $TT^*$ argument to finally bound $\|W_{R,\lambda} \|_{L^2(X) \to L^p(X)}$, and combine this with a lower bound on the spectral action of $W_{R,\lambda}$ to obtain our deterministic result. (Lemma \ref{lem: W_Tlambda norm bound} and Theorem \ref{thm: MaxSpectralGapMainResult})
\end{enumerate}

\subsection{Proof of Theorem \ref{thm:Main} for optimal spectral gap surfaces}
We begin by constructing a family of integral operators to analyse the eigenfunctions of the Laplacian. To this end, we define for $t\geq 0$ and $r$ in the same range as the $s_\lambda\in\C$, the functions $j_t$ given by
	\begin{align*}
		j_t(r)=\frac{\cos(rt)}{\sqrt{\cosh\left(\frac{\pi r}{2}\right)}}.
	\end{align*}
Using the Selberg transform, one may associate to $j_t$ a radial kernel $\ell_t(z,w)=\ell_t(d(z,w))$ for an integral operator $P_t$ acting on functions of $\H$ given by
	\begin{align*}
		P_tf(z)=\int_\H \ell_t(z,w)f(w)\df\mu(w).
	\end{align*}
The kernel $\ell_t$ is in fact real valued, which can be seen by the fact that the Selberg transform of the complex conjugate of $\ell_t$ coincides with $j_t$, since $j_t$ is real valued for the specified $r$. Indeed,
	\begin{align*}
		\overline{j_t(r)}&=\int_{-\infty}^\infty e^{irv}\int_{|v|}^\infty \overline{\ell_t(\rho)}\frac{\sinh(\rho)}{\sqrt{\cosh(\rho)-\cosh(v)}}\df\rho\df v.
	\end{align*}  
The formal adjoint of $P_t$ then takes the form	
	\begin{align*}
		P_t^*f(z)=\int_\H \ell_t(z,w)f(w)\df\mu(w),
	\end{align*}
since the kernel $\ell_t$ is real and symmetric in $z$ and $w$. These operators may then be defined on the surface via the fundamental domain and using the automorphic kernel formed from the group $\Gamma$ generating the surface $X$ as described in the previous section to give
	\begin{align*}
		P_tf(z)&=\int_D\sum_{\gamma\in\Gamma} \ell_t(z,\gamma w)f(w)\df\mu(w),
		\shortintertext{and}
		P_t^*f(z)&=\int_D\sum_{\gamma\in\Gamma} \ell_t(z,\gamma w)f(w)\df\mu(w).
	\end{align*}
On the surface, these operators are then bounded as operators from $L^2(X)$ to $L^p(X)$ and from $L^q(X)$ to $L^2(X)$ respectively with $p>2$ and $p$ and $q$ conjugate indices using the decay conditions on the $j_t$.\par 
In Section \ref{sec: ArbitrarySpecGap} we will see that the desired result in fact holds trivially for the constant eigenfunctions, thus we will only use $P_t$ to analyse the eigenfunctions corresponding to the eigenvalues away from zero. Thus, when testing our operator against an arbitrary function, we will remove the component of the function corresponding to the zero eigenspace. To this end, we then define the operator $\Pi \colon L^q(X)\to L^q(X)$ by
	\begin{align*}
		f\mapsto f- \dashint_D f(z) \df\mu(z) \eqqcolon f - \bar{f},
	\end{align*}
where $ \dashint$ denotes the average:
$$ \dashint_D f(z) \df\mu(z) = \frac1{\mathrm{Vol}(D)} \int_D f(x) \df\mu(z).$$
Next we begin to understand the pertinent properties of the operators $P_t$. One crucial property that they possess is a linearisation formula under composition with their adjoint. 
\begin{lemma}
\label{lem: LinearisationResult}
The integral kernel of the composition operator $P_tP_s^* \colon L^q(X)\to L^p(X)$ for $t,s\geq 0$ is given by
	\begin{align*}
		\frac{1}{2} \left(k_{t+s}+k_{|t-s|} \right),
	\end{align*}
where $k_t$ is the associated radial kernel through the Selberg transform with the function
	\begin{align*}
		h_t(r)= \frac{\cos(rt)}{\cosh(\frac{\pi r}{2})}.
	\end{align*}
In particular, if $Q_t \colon L^q(X)\to L^p(X)$ is the associated integral operator for the kernel $k_t$, then
	\begin{align}\label{e:linearisation}
		P_tP_s^*\Pi= \frac{1}{2} \left(Q_{t+s}\Pi+Q_{|t-s|}\Pi \right).
	\end{align}
\end{lemma}

\begin{proof} 
This is essentially a consequence of trigonometric relations of the cosine function. Notice firstly that the kernel of $P_tP_s^*$ on $\H$ is given by the convolution kernel 
	\begin{align*}
		m_{t,s}(z,w)=\int_\H \ell_t(z,w')\ell_s(w,w')\df\mu(w'),
	\end{align*}
which is itself a radial kernel by invariance of the measure $\df\mu$ under isometries. Let $M_{t,s}(d(z,w))=m_{t,s}(z,w)$ denote the associated function on $\R$ that generates $m_{t,s}$. By Theorem \ref{thm: BergeronThmOnSelberg}, for an eigenfunction $\psi$ of the Laplacian on $\H$ with corresponding eigenvalue $\lambda=\frac{1}{4}+r^2$, for $r$ the eigenvalue parameter from before, we obtain
	\begin{align*}
		P_tP_s^*\psi=\mathcal{S}(M_{t,s})(r)\psi,
	\end{align*}
where $\cS(M_{t,s})$ denotes the Selberg transform of the function $M_{t,s}$. On the other hand, by applying each of the operators in turn,
	\begin{align*}
		P_tP_s^*\psi= j_t(r)j_s(r)\psi,
	\end{align*}
and hence 
	\begin{align*}
		j_t(r)j_s(r)=\mathcal{S}(M_{t,s})(r).
	\end{align*}
Notice then for real $r$, that one has
	\begin{align*}
		j_t(r)j_s(r)&=\frac{\cos(r(t+s))}{2\cosh(\frac{\pi r}{2})}+\frac{\cos(r|t-s|)}{2\cosh(\frac{\pi r}{2})}\\
		&=\frac{1}{2}(h_{t+s}(r)+h_{|t-s|}(r)).
	\end{align*}
Similarly, when $r=bi$ for $b\in[0,\frac{1}{2}]$, we obtain
	\begin{align*}
		j_t(r)j_s(r) &=\frac{\cosh(b(t+s))}{2\cos\left(\frac{\pi b}{2}\right)}+\frac{\cosh(b(t-s))}{2\cos\left(\frac{\pi b}{2}\right)}\\
		&=\frac{1}{2}(h_{t+s}(r)+h_{|t-s|}(r)).
	\end{align*}
By applying the inverse Selberg transform, it follows that $m_{t,s}=\frac{1}{2}(k_{t+s}+k_{|t-s|})$, where $k_t$ is as given in the statement of the lemma, and therefore
	\begin{align*}
		P_tP_s^*= \frac{1}{2} \left(Q_{t+s}+Q_{|t-s|} \right).
	\end{align*} 
By composing with $\Pi$, we obtain \eqref{e:linearisation}.
\end{proof}

We remark that the function $h_t(r)$ in the previous lemma is precisely the Selberg transform considered by Brooks and Lindenstrauss~\cite{BL14}. It was introduced previously in the article of Iwaniec and Sarnak~\cite{IS95}, where its Fourier transform was used to define a kernel to obtain sup norm bounds of eigenfunctions of the Laplacian on arithmetic surfaces. Thus much is already known regarding estimates on the kernel induced by this function through the Selberg transform. Indeed, Brooks and Lindenstrauss~\cite{BL14} have obtained the following bounds that are crucial in our investigation.

\begin{lemma}[Brooks and Lindenstrauss~\cite{BL14}] 
\label{lem: BrooksLindenstraussBounds}
With $k_t$ as above denoting the kernel associated via the Selberg transform with the function $h_t$, we have the following estimates. A sup norm bound of
\begin{align}
	\|k_t\|_\infty \lesssim e^{-t/2},
\end{align}
and rapid decay outside a ball of radius $4t$ of the type
\begin{align}
	\int_{4t}^\infty |k(\rho)| \sinh (\rho) \, \df\rho \lesssim e^{-t}.
\end{align}
\end{lemma}

Next we consider the operator $Q_t$ as defined in Lemma \ref{lem: LinearisationResult}. We combine the bounds of Lemma \ref{lem: BrooksLindenstraussBounds} with the condition \eqref{eqn: ShortGeodesicBound} assumed of our surfaces to obtain suitable bounds on the operator norm of $Q_t\Pi$ in terms of the parameter $t$.

\begin{lemma}
\label{lem: QtBounds}
Suppose that $Q_t$ and $\Pi$ are defined as above. For $t\leq \frac{R}{4}$, with $R$ as in \eqref{eqn: ShortGeodesicBound}, one may bound the $L^q(X)\to L^p(X)$ operator norm by
\begin{align*}
	\|Q_t \Pi \|_{L^q(X)\to L^p(X)} \lesssim_\delta C(X)e^{-\alpha_p t},
\end{align*}
where $\alpha_p$ can be chosen to equal $(\frac12 - \delta)(1-\frac2p)$ for any $\delta >0$.
\end{lemma}

\begin{proof}
We will proceed by interpolation, first calculating the norm $\|Q_t \Pi \|_{L^1(X)\to L^\infty(X)}$. We have by the definition of the automorphic kernel integral operator that 
	\begin{align*}
		\|Q_t \|_{L^1(X)\to L^\infty(X)} \leq \sup_{z,w \in D} \sum_{\gamma \in \Gamma} |k_t(d(z, \gamma w))|.
	\end{align*}
This summation can then be split into two parts corresponding to propagation at times shorter and longer than $4t$. Indeed, for any $z,w \in D$,
\begin{align*}
	\sum_{\gamma \in \Gamma} |k_t(d(z, \gamma w))| \leq \sum_{\gamma \in \Gamma} |k_t(d(z, \gamma w))|\mathbf{1}_{[0,4t]}(d(z,\gamma w)) + \int_{4t}^\infty |k_t(\rho)| e^\rho \, \df\rho,
\end{align*}
with the latter integral arising from the fact that $|\{\gamma\in\Gamma:m\leq d(z,\gamma w)\leq m+1\}|=O(e^m)$, by a simple counting argument of the number of fundamental domains intersecting a ball of radius $m$.

By Lemma \ref{lem: BrooksLindenstraussBounds}, we then obtain
\begin{align*}
	\sum_{\gamma \in \Gamma} |k_t(d(z, \gamma w))| \lesssim  | \{ \gamma \in \Gamma \, | \, d(z,\gamma w) \leq 4t \}| \, e^{-t/2} + e^{-t}.
\end{align*}
Condition \eqref{eqn: ShortGeodesicBound} then asserts that for $t \leq R/4$ we have
\begin{align*}
	| \{ \gamma \in \Gamma \, | \, d(z,\gamma w) \leq 4t \}| \lesssim_\delta C(X)e^{\delta t},
\end{align*}
for any $\delta > 0$. This yields
\begin{align*}
	\sum_{\gamma \in \Gamma} |k_t(d(z, \gamma w))| \lesssim_\delta C(X) e^{-t(\frac12 - \delta)},
\end{align*}
and hence
\begin{align*}
	\|Q_t \|_{L^1(X)\to L^\infty(X)} \lesssim  e^{-t(\frac12 - \delta)}.
\end{align*}
Incorporating the operator $\Pi$, we then obtain
\begin{align*}
	\|Q_t \Pi f\|_\infty &= \|Q_t (f - \bar f) \|_\infty\\
	&\leq  \|Q_t \|_ {L^1 \to L^\infty} ( \|f\|_1 + \|\bar f\|_1)\\
	&\lesssim_\delta C(X)e^{-t(\frac{1}{2}-\delta)} \|f\|_1,
\end{align*}
and thus we get the bound
\begin{align*}
	\|Q_t \Pi\|_{L^1 \to L^\infty} \lesssim_\delta C(X) e^{-t(\frac{1}{2}-\delta)}.
\end{align*}

Next we calculate the $L^2(X)\to L^2(X)$ norm.  We note that the operators $Q_t$ and $\Pi$ acting on $L^2(X)$ to $L^2(X)$ are both self-adjoint. Indeed, the former has a real and symmetric kernel and the latter is a projection. In addition, the operators $Q_t$ and $\Pi$ commute with each other since for any $f\in L^2(X)$,
	\begin{align*}
		\Pi Q_tf(z)&= Q_tf(z)-\dashint_D Q_tf(w)\,\df\mu(w)\\
		&=Q_tf(z)-\frac{1}{\Vol(D)}\int_D f(w')\int_D\sum_{\gamma\in\Gamma} k_t(w,\gamma w')\,\df\mu(w)\df\mu(w')\\
		&=Q_tf(z)-h_t\left(\frac{1}{2}i\right)\bar{f}\\
		&=Q_t(f-\bar{f})(z)\\
		&=Q_t\Pi f(z).
	\end{align*}
This means that $Q_t\Pi$ is a self-adjoint operator from $L^2(X)$ to $L^2(X)$ and its norm is equal to its spectral radius. It follows from the projection operator, Theorem \ref{thm: BergeronThmOnSelberg} and the fact that $X$ has optimal spectral gap, the norm is given by
	\begin{align*}
		\|Q_t\Pi\|_{L^2(X)\to L^2(X)}=\sup_{r\in [0,\infty)} |h_t(r)| \leq 1.
	\end{align*}
Finally, we apply the Riesz-Thorin interpolation theorem to get the desired bound.
\end{proof}

We now construct an operator specific to an eigenvalue $\lambda\geq\frac{1}{4}$ of the Laplacian of $X$.  To do this, we wish to combine our propagators $P_t$ along values of $t$ for which the bounds obtained in Lemma \ref{lem: QtBounds} are valid. In doing so, we are able to exhibit the dependence upon the parameter $R$ of the surface. To this end, fix $T\leq \frac{1}{8}R$ and let $W_{T,\lambda} \colon L^2(X)\to L^p(X)$ to be the operator defined for any $p\geq 2$ by	
	\begin{align}
	\label{eq: W_Tlambda definition}
		W_{T,\lambda}f(z)=\int_0^T \cos(s_\lambda t)P_t\Pi f(z) \, \df t,
	\end{align}
 where $s_\lambda$ is the spectral parameter in the parametrisation $\lambda=s_\lambda^2+\frac{1}{4}$ of the eigenvalue.
 
To calculate the $L^2(X)\to L^p(X)$ operator norm we will employ a $TT^*$ argument, that is we use the fact that
	\begin{align*}
		\|W_{T,\lambda}\|_{L^2(X)\to L^p(X)}^2=\|W_{T,\lambda}W_{T,\lambda}^*\|_{L^q(X)\to L^p(X)},
	\end{align*}
where $q$ is the conjugate index of $p$. 
\begin{lemma}
\label{lem: W_Tlambda norm bound}
Let $\lambda\geq\frac{1}{4}$ be an eigenvalue of $\Delta$ on $X$ and fix $T\leq \frac{1}{8}R$ where $R$ is as in \eqref{eqn: ShortGeodesicBound}. If $W_{T,\lambda}$ is defined as in \eqref{eq: W_Tlambda definition}, then 
	\begin{align*}
		\|W_{T,\lambda}\|_{L^2(X)\to L^p(X)} \lesssim_{p} \sqrt{C(X)T}.
	\end{align*}
\end{lemma}
\begin{proof}
We compute through an application of Minkowski's integral inequality that
	\begin{align*}
		\|W_{T,\lambda}W_{T,\lambda}^*\|_{L^q(X)\to L^p(X)}&=\left\|\int_0^T\int_0^T \cos(s_\lambda t)\cos(s_\lambda s)P_t\Pi P_s^* \, \df s\df t\right\|_{L^q(X)\to L^p(X)}\\
		&\leq \int_0^T\int_0^T \|P_t\Pi P_s^*\|_{L^q(X)\to L^p(X)} \, \df s\df t.
	\end{align*}
It thus suffices to consider the norm $\|P_t\Pi P_s^*\|_{L^q(X)\to L^p(X)}$. 

Notice that by a similar argument to that used in the proof of Lemma \ref{lem: QtBounds}, we can see that $\Pi$ commutes with the adjoint $P_s^*$. We then use Lemma \ref{lem: LinearisationResult} to deduce that 
	\begin{align*}
		\|P_t\Pi P_s^*\|_{L^q(X)\to L^p(X)}\lesssim \|Q_{t+s}\Pi\|_{L^q(X)\to L^p(X)}+\|Q_{|t-s|}\Pi\|_{L^q(X)\to L^p(X)}.
	\end{align*}
Now since $t+s\leq 2T\leq \frac{1}{4}R$, it follows from Lemma \ref{lem: QtBounds} that
	\begin{align*}
		\|P_t\Pi P_s^*\|_{L^q(X)\to L^p(X)}&\lesssim_\delta C(X)(e^{-\alpha_p (t+s)} +e^{-\alpha_p |t-s|})\\
		&\lesssim_\delta C(X) e^{-\alpha_p |t-s|}.
	\end{align*}
Taking $\delta$ sufficiently small so that $\alpha_p>0$ one may substitute this bound back into the integral to obtain
	\begin{align*}
		\|W_{T,\lambda}W_{T,\lambda}^*\|_{L^q(X)\to L^p(X)}&\lesssim C(X)\int_0^T\int_0^T e^{-\alpha_p |t-s|} \: \df s\df t\\
		&=C(X)\int_0^T\int_0^te^{-\alpha_p(t-s)} \: \df s\df t+\int_0^T\int_t^Te^{-\alpha_p(s-t)} \: \df s\df t\\
		&\lesssim_p C(X)T.
	\end{align*}	
Notice that the dependence of the implicit constant on $\delta$ is now removed due to the fixing of a given $\delta>0$. The bound
	\begin{align*}
		\|W_{T,\lambda}\|_{L^2(X)\to L^p(X)}\lesssim_p \sqrt{C(X)T}
	\end{align*}
is then immediate.
\end{proof}
With this upper bound, we turn to examining the spectral action of $W_{T,\lambda}$ on an eigenfunction with eigenvalue $\lambda$. For this, we use the explicit form of the Selberg transform to obtain our desired result.
\begin{thm}
\label{thm: MaxSpectralGapMainResult}
Suppose that $X$ is a compact hyperbolic surface such that $\sigma_X(\Delta)\subseteq\{0\}\cup[\frac{1}{4},\infty)$. If $\psi_\lambda$ is an eigenfunction of $\Delta$ with eigenvalue $\lambda\geq\frac{1}{4}$, then 
	\begin{align*}
		\|\psi_\lambda\|_p\lesssim_{\lambda, p} \frac{\sqrt{C(X)}}{\sqrt{R}}\| \psi_\lambda \|_2,
	\end{align*}
where $R$ and $C(X)$ are given by condition \eqref{eqn: ShortGeodesicBound}. 
\end{thm}
\begin{proof}
We consider the action of the test operator $W_{T,\lambda}$, given by \eqref{eq: W_Tlambda definition}, on $\psi_\lambda$, but with $T=\frac{1}{8}R$. By Lemma \ref{lem: W_Tlambda norm bound}, one immediately obtains
	\begin{align*}
		\|W_{T,\lambda}\psi_\lambda\|_p\leq \|W_{T,\lambda}\|_{L^2(X)\to L^p(X)}\|\psi_\lambda\|_2\lesssim_p \sqrt{C(X)T}\|\psi_\lambda\|_2.
	\end{align*}
On the other hand, applying Theorem \ref{thm: BergeronThmOnSelberg} provides that
	\begin{align*}
		\|W_{T,\lambda}\psi_\lambda\|_p=\frac{1}{\sqrt{\cosh\left(\frac{\pi s_\lambda}{2}\right)}}\int_0^T\cos^2(s_\lambda t)\,\df t\|\psi_\lambda\|_p\gtrsim_\lambda T\|\psi_\lambda\|_p.
	\end{align*}
Dividing through then gives
	\begin{align*}
		\|\psi_\lambda\|_p\lesssim_{\lambda,p} \frac{\sqrt{C(X)}}{\sqrt{R}}\|\psi_\lambda\|_2.
	\end{align*}
\end{proof}

\section{Deterministic Bounds for Surfaces with an Arbitrary Spectral Gap}
\label{sec: ArbitrarySpecGap}

We now consider the case where the spectrum of the Laplacian on the compact hyperbolic surface $X$ takes values in the full range $[0,\infty)$. To deal with this, we utilise two separate methods for the eigenfunctions belonging to the different parts of the spectrum. 

For the untempered spectrum, we demonstrate a far stronger bound on the norms of eigenfunctions than previously obtained in the optimal spectral gap case above. Indeed, we show that the norm has some exponential decay in the parameter $R$ given by \eqref{eqn: ShortGeodesicBound}. This is carried out via a rescaled ball averaging operator of functions on the surface, which was previously used by Le Masson and Sahlsten~\cite{LMS17}. The pertinent information required here is the spectral action of this operator on eigenfunctions, which is given through the Selberg transform.\par 

For the portion of the spectrum lying above $\frac{1}{4}$, we may use an identical technique to the optimal spectral gap case to obtain the relevant bounds. However, due to the introduction of eigenfunctions in the untempered spectrum the result is weakened slightly and is only valid for values of $p$ bounded below by a function dependent on the spectral gap of the surface. We begin by providing an outline of the proof.

\subsection{Outline of proof} 
The methodology for the proof is similar to that in the optimal spectral gap case, so we emphasise the main differences.

\begin{enumerate}
	\item Firstly we show the stronger exponential decay result for the $L^p$ norms of the untempered portion of the spectrum. This is done via a rescaled averaging operator over hyperbolic balls on the surface to obtain the $L^\infty$ norm and then a simple interpolation of this with the trivial $L^2$ norm bound provides the result for general $p>2$. (Theorem \ref{thm: UntemperedSupNormDecay} and Corollary \ref{cor: untempered p-norm decay})
	\item For the tempered portion of the spectrum, we utilise the same method as in Section \ref{sec: MaxSpecGap}. The main difference is that the existence of untempered eigenfunctions, other than constants, put restrictions upon the values of $p$ for which the bounds are valid dependent upon the spectral gap. These come from a technicality in the computation of the $L^2\to L^2$ operator norm of the propagation operator since the convolution operator eigenvalue for untempered eigenfunctions of the Laplacian exhibits exponential growth in the propagation parameter. (Theorem \ref{thm: arbitrary spectral gap result})
\end{enumerate}

\subsection{Untempered eigenfunctions deterministic bound proof}
We start by defining the required ball averaging operator on the surface. Let $(B_t)_{t\geq 0}$ denote the family of operators 
	\begin{align*}
		B_tf(z)= \frac{1}{\sqrt{\cosh(t)}}\int_{B(z,t)} f(w)\df\mu(w),
	\end{align*}
acting on appropriate functions of $\H$. We pass this to an operator on the surface $X=\Gamma\backslash\H$ by considering functions defined upon a fundamental domain $D$ and using the automorphic kernel, so that
	\begin{align}
	\label{eq: B_t def}
		B_tf(z) = \frac{1}{\sqrt{\cosh(t)}}\int_D \sum_{\gamma\in\Gamma} \1_{\{d(z,\gamma w)<t\}} f(w)\df\mu(w).
	\end{align}
It then follows immediately that the kernel of this operator is induced by the function
	\begin{align*}
		k_t(\rho)=\frac{\1_{\{\rho<t\}}}{\sqrt{\cosh(t)}},
	\end{align*}
whose Selberg transform is given by
	\begin{align*}
		\mathcal{S}(k_t)(r) = 4\sqrt{2}\int_0^t \cos(ru)\sqrt{1-\frac{\cosh(u)}{\cosh(t)}} \, \df u.
	\end{align*}
We now prove the required bounds in order to deduce our desired result for eigenfunctions in the untempered spectrum. In doing so, we complete the result for the optimal spectral gap case in the previous section, since we then have the required bound for the constant eigenfunctions.  We initially prove a slightly stronger result than required, namely that a real linear combination of real-valued Laplacian eigenfunctions corresponding to eigenvalues in the untempered spectrum have strong sup norm decay. The case of an arbitrary untempered eigenfunction follows immediately from a simplification of the proof of this result.
\begin{thm}
\label{thm: UntemperedSupNormDecay}
Let $\varepsilon>0$ and suppose that
	\begin{align*}
		f=\sum_{j=1}^n\alpha_j\psi_j
	\end{align*}
is a finite real linear combination of mutually orthogonal untempered real-valued eigenfunctions $\{\psi_j\}_{j=1}^n$ of the Laplacian with corresponding eigenvalues $\{\lambda_j\}_{j=1}^n\subseteq [0,\frac{1}{4}-\varepsilon)$. Then for any $\delta> 0$
	\begin{align*}
		\|f\|_\infty \lesssim_\delta \frac{C(X)}{e^{(\sqrt{\varepsilon}-\delta)R}-1}\|f\|_2
	\end{align*}
where $R$ and $C(X)$ are given by \eqref{eqn: ShortGeodesicBound}.
\end{thm}

\begin{proof}	
The eigenfunctions are smooth so it follows that $f$ is smooth.  The compactness of the surface gives that there exists $x\in D$ such that $|f(x)|=\|f\|_\infty$. Without loss of generality we can assume that $f(x)>0$. Moreover, we can assume that each $\alpha_j\psi_j(x)>0$, since removing negative terms increases the value of $f(x)$ (and hence $\|f\|_\infty$, albeit now potentially attained at a point different to $x$) whilst decreasing $\|f\|_2$ by orthogonality,
	\begin{align*}
		\left\|\sum_{j\in I\subset \{1,\ldots,n\}}\alpha_j\psi_j\right\|_2^2&=\sum_{j\in I\subset \{1,\ldots,n\}}\|\alpha_j\psi_j\|_2^2\leq \sum_{j=1}^n \|\alpha_j\psi_j\|_2^2=\left\|\sum_{j=1}^n\alpha_j\psi_j\right\|_2^2.
	\end{align*}
We now consider the ball-averaging operators defined in \eqref{eq: B_t def} for radii $t\leq R$, where $R$ is given by the surface assumption \eqref{eqn: ShortGeodesicBound}. The fact $t\leq R$ means that by definition the number of terms in the summation in the automorphic kernel of $B_t$ is bounded by $e^{\gamma t}$ for any $\gamma>0$. 
 
We now use the Selberg transform of the associated kernel function of $B_t$ to analyse the action of $B_t$ on $f$ about the point $x$,
	\begin{align}
	\label{eq: B_t action on f}
		B_tf(x)=\sum_{j=1}^n \mathcal{S}(k_t)(s_{\lambda_j}i)\alpha_j\psi_j(x),
	\end{align}
where $s_{\lambda_j}\in [\sqrt{\varepsilon},\frac{1}{2}]$ is the eigenvalue parameter of $\lambda_j$, where $\lambda_j = \frac{1}{4} - s_{\lambda_j}^2$. We now demonstrate that the values of $\mathcal{S}(k_t)(s_{\lambda_j}i)$ are in fact non-negative and bounded below for large enough $t$ (and hence large enough $R$ in \eqref{eqn: ShortGeodesicBound}) by an exponentially growing term. Notice that
	\begin{align*}
		\mathcal{S}(k_t)(s_{\lambda_j}i)&\gtrsim \int_0^t\cos(s_{\lambda_j}iu)\sqrt{1-\frac{\cosh(u)}{\cosh(t)}} \, \df u\\
		&=\int_0^t \cosh(s_{\lambda_j}u)\sqrt{1-\frac{\cosh(u)}{\cosh(t)}} \, \df u,
	\end{align*}
and hence the values are non-negative. Moreover, we may use the fact that the term underneath the square root is bounded by 1 and hence the integral is bounded below by the integral without the square root which can be explicitly calculated. Thus,
	\begin{align*}
		\mathcal{S}(k_t)(s_{\lambda_j}i) &\gtrsim \int_0^t \cosh(s_{\lambda_j} u) \, \df u -\frac{1}{2}\int_0^t \frac{\cosh((s_{\lambda_j}+1)u)}{\cosh(t)}+\frac{\cosh((s_{\lambda_j}-1)u)}{\cosh(t)} \, \df u\\		
		&=\frac{\sinh(s_{\lambda_j}t)}{s_{\lambda_j}}-\frac{1}{2}\left(\frac{\sinh((s_{\lambda_j}+1)t)}{(s_{\lambda_j}+1)\cosh t}+\frac{\sinh((s_{\lambda_j}-1)t)}{(s_{\lambda_j}-1)\cosh t}\right).\\
	\end{align*}
It is easy to see that this expression increases for all values of $t$ in the parameter $s_{\lambda_j}$ and hence we may bound this expression below with $s_{\lambda_j}$ replaced by $\sqrt{\varepsilon}$. In addition, one may observe for large enough $t$ that this expression is bounded below by $\sinh(\sqrt{\varepsilon} t)$ and in fact the size of $t$ required for this bound may be taken to be uniform in $\varepsilon$ (it suffices to take $t\geq 3$). Thus for $t\geq 3$ and each $j=1,\ldots,n$,
	\begin{align*}
		\mathcal{S}(k_t)(s_{\lambda_j}i)  \gtrsim\sinh(\sqrt{\varepsilon}t).
	\end{align*}
In particular, using the non-negativity of each term in \eqref{eq: B_t action on f}, we thus obtain for $t\geq 3$ that
	\begin{align*}
		B_tf(x)\gtrsim \sinh(\sqrt{\varepsilon}t)f(x)=\sinh(\sqrt{\varepsilon}t)\|f\|_\infty.
	\end{align*}
Conversely, notice for $t\leq R$ that there is at most $C(X)C_0(\delta)e^{\delta t}$ non-zero terms in the summation for the automorphic kernel of $B_t$ for any $\delta>0$, and hence we have 
	\begin{align*}
		\left\|\sum_{\gamma\in\Gamma}\frac{\1_{\{d(x,\gamma \cdot)\leq t\}}}{\sqrt{\cosh(t)}}\right\|_2^2&\leq \int_D\sum_{\gamma\in\Gamma}\frac{\1_{\{d(x,\gamma w)\leq t\}}}{\cosh(t)} \, \df\mu(w)\\
		&\lesssim_\delta \frac{C(X)e^{\delta t}}{\cosh(t)}\Vol(\text{Ball of radius $t$})\lesssim_\delta C(X)e^{\delta t}.
	\end{align*}
By the Cauchy-Schwarz inequality, we obtain for all $\delta>0$ that
	\begin{align*}
		|B_tf(x)|&=\left|\int_D\sum_{\gamma\in\Gamma} \frac{\1_{\{d(x,\gamma w)\leq t\}}}{\sqrt{\cosh(t)}}f(w) \, \df\mu(w)\right|\\
		&\leq \left\|\sum_{\gamma\in\Gamma}\frac{\1_{\{d(x,\gamma \cdot)\leq t\}}}{\sqrt{\cosh(t)}}\right\|_2\|f\|_2\\
		&\lesssim_\delta C(X)e^{\delta t}\|f\|_2.
	\end{align*}
We may then combine the two inequalities so that for any $\delta>0$ and $3\leq t\leq R$,
	\begin{align*}
		\|f\|_\infty \lesssim_\delta \frac{C(X)e^{\delta t}}{\sinh(\sqrt{\varepsilon}t)}\|f\|_2.
	\end{align*}
It follows that 
	\begin{align*}
		\|f\|_\infty \lesssim_\delta \frac{C(X)}{e^{(\sqrt{\varepsilon}-\delta)t}-1}\|f\|_2,
	\end{align*}
and taking $t=R$ then gives the result.
\end{proof}

By using the same argument as in the above proof and applying an interpolation argument on the norms, we obtain the desired eigenfunction bound for any eigenfunctions corresponding to an eigenvalue in the untempered spectrum.
\begin{cor}
\label{cor: untempered p-norm decay}
Suppose that $\psi_\lambda$ is an eigenfunction of the Laplacian for the surface $X$ with eigenvalue $\lambda\in [0,\frac{1}{4})$. Then for any $\varepsilon>0$ for which $\lambda\in [0,\frac{1}{4}-\varepsilon)$ we have 
	\begin{align*}
		\|\psi_\lambda\|_p \lesssim_\delta \frac{C(X)}{(e^{(\sqrt{\varepsilon}-\delta)R}-1)^{1-\frac{2}{p}}}\|\psi_\lambda\|_2,
	\end{align*}
for any $\delta >0$, where $R$ is given by \eqref{eqn: ShortGeodesicBound}.
\end{cor}

\begin{proof}
Once again, by compactness of $D$ there exists some $x\in D$ for which $|\psi_\lambda(x)|=\|\psi_\lambda\|_\infty$. Using the ball averaging operator then gives that
	\begin{align*}
		|B_t\psi_\lambda(x)|=|\mathcal{S}(k_t)(s_\lambda i)||\psi_\lambda(x)|.
	\end{align*}
For $t\geq 3$, we obtain as in Theorem \ref{thm: UntemperedSupNormDecay} that 
	\begin{align*}
		|B_t\psi_\lambda(x)|\geq\sinh(\sqrt{\varepsilon}t)\|\psi_\lambda\|_\infty.
	\end{align*}
Analysing the upper bound as before then results in
	\begin{align*}
		\|\psi_\lambda\|_\infty \lesssim_\delta \frac{C(X)}{e^{(\sqrt{\varepsilon} - \delta)R}-1}\|\psi_\lambda\|_2,
	\end{align*}
for any $\delta > 0$.
We now use interpolation to see that
	\begin{align*}
		\|\psi_\lambda\|_p &\leq \|\psi_\lambda\|_2^{\frac{2}{p}}\|\psi_\lambda\|_\infty^{1-\frac{2}{p}}\\
		&\lesssim_\delta \frac{C(X)}{(e^{(\sqrt{\varepsilon}-\delta)R}-1)^{1-\frac{2}{p}}}\|\psi_\lambda\|_2.
	\end{align*}
\end{proof}

\subsection{Proof of Theorem \ref{thm:Main}}

For the tempered eigenfunctions, we can use the same method as in the optimal spectral gap case.  The smaller spectral gap associated with the surface weakens the values of $p$ for which the result holds, however at worst we obtain that the bounds are valid for at least $p>4$.

\begin{thm}
\label{thm: arbitrary spectral gap result}
Suppose that $X$ is a compact hyperbolic surface whose smallest non-zero eigenvalue of the Laplacian is at least $\frac{1}{4}-\beta^2$ for some $\beta\in[0,\frac{1}{2})$. For a tempered eigenfunction $\psi_\lambda$ of the Laplacian with eigenvalue $\lambda\geq \frac{1}{4}$, we have the following bound
	\begin{align*}
		\|\psi_\lambda\|_p \lesssim_{p,\lambda} \frac{\sqrt{C(X)}}{\sqrt{R}}\|\psi_\lambda\|_2,
	\end{align*}
for any $2+4\beta<p\leq \infty$, with $R$ as in \eqref{eqn: ShortGeodesicBound}.
\end{thm}

\begin{proof}
We utilise the operator $W_{T,\lambda}$ as defined by \eqref{eq: W_Tlambda definition}. As in Lemma \ref{lem: W_Tlambda norm bound}, the calculation of the $L^2(X)\to L^p(X)$ norm of $W_{T,\lambda}$ is reduced to computing the operator norms 
	\begin{align*}
		\|Q_t\Pi\|_{L^1(X)\to L^\infty(X)}\ \ \ \text{and}\ \ \ \|Q_t\Pi\|_{L^2(X)\to L^2(X)},
	\end{align*}
where $Q_t$ is the operator defined in Lemma \ref{lem: LinearisationResult}. Using the same argument as in Lemma \ref{lem: QtBounds}, with $t\leq R/4$ we obtain that
	\begin{align*}
		\|Q_t\Pi\|_{L^1(X)\to L^\infty(X)} \lesssim_\delta C(X) e^{-t(\frac{1}{2}-\delta)},
	\end{align*}
for any $\delta >0$. For the $L^2(X)\to L^2(X)$ norm, we notice that in this case there is an exponential growth in the spectral radius. Indeed, we now have 
	\begin{align*}
		\|Q_t\Pi\|_{L^2(X)\to L^2(X)} = \sup_{\substack{r\in [0,\infty),\\ \text{or } r=ai,\ a\in [0,\beta]}}\left|\frac{\cos(rt)}{\cosh(\pi r/2)}\right|\leq e^{\beta t}.
	\end{align*}
Applying the Riesz-Thorin Interpolation Theorem, we then obtain for the conjugate exponent $q$ of $p$ and any $\delta>0$ that
	\begin{align*}
		\|Q_t\Pi\|_{L^q(X)\to L^p(X)} &\leq \|Q_t\Pi\|_{L^1(X)\to L^\infty(X)}^{1-\frac{2}{p}}\|Q_t\Pi\|_{L^2(X)\to L^2(X)}^\frac{2}{p}\\
		&\lesssim_\delta C(X)e^{-t(\frac{1}{2}-\delta-\frac{1}{p}+\frac{2}{p}\delta-\beta\frac{2}{p})}.
	\end{align*}
When $p>2+\frac{4\beta}{1-\delta}$ (assuming $\delta<1$), the norm exhibits exponential decay. Since this is true for all $0<\delta<1$, it follows that there is exponential decay whenever $p>2+4\beta$ and in this case, we can show as in Lemma \ref{lem: W_Tlambda norm bound} that 
	\begin{align*}
		\|W_{T,\lambda}\|_{L^2(X)\to L^p(X)}\lesssim_p \sqrt{C(X)T}.
	\end{align*}
Since the spectral action of $W_{T,\lambda}$ on $\psi_\lambda$ is identical to that considered in Theorem \ref{thm: MaxSpectralGapMainResult}, we also recover the lower bound
	\begin{align*}
		\|W_{T,\lambda}\psi_\lambda\|_p\gtrsim_\lambda T\|\psi_\lambda\|_p.
	\end{align*}
Combining these two estimates gives the desired result.
\end{proof}
Theorem \ref{thm:Main} is then obtained by combining Theorem \ref{thm: MaxSpectralGapMainResult}, Corollary \ref{cor: untempered p-norm decay} and Theorem \ref{thm: arbitrary spectral gap result}.

\section{Teichm\"{u}ller Theory and Random Surfaces}\label{s:randombackground}
This section gathers much of the background required and notation utilised when formulating and working with probabilistic statements on surfaces in this paper. Further details on the foundational material on Teichm\"{u}ller theory, geodesics and mapping class groups can be found in \cite{IT12}, \cite{Bus10} and \cite{FM11}. \par 
Let $g,n\geq 0$ be integers. We will denote by $\Sigma_{g,n}$ a surface of genus $g$ with $n$ boundary components; if $n=0$ this is simply written as $\Sigma_g$. Given the $n$ boundary components, one can associate a length vector $L=(L_1,\ldots,L_n)\in\R^n_{\geq 0}$ to the surface such that the $i^\text{th}$ boundary component has length $L_i$. If $L_i=0$, then the component is thought of as a cusp or marked point on the surface.\par 
The \textit{Teichm\"{u}ller space} of signature $(g,n)$ and length vector $L\in\R^n_{\geq 0}$ is defined to be the space
	\begin{align*}
		\mathcal{T}_{g,n}(L)= \left\{(X,f): \parbox{10cm}{$X$ is a complete hyperbolic surface of genus $g$ and with $n$ geodesic boundary components with lengths corresponding to $L$ and $f:\Sigma_{g,n}\to X$ is a diffeomorphism.}\right\}\Big{/}\sim,
	\end{align*}
where $\sim$ is the equivalence relation defined by $(X,f)\sim (Y,g)$ if and only if there exists an isometry $h:X\to Y$ for which
	\begin{align*}
		g^{-1}\circ h\circ f \colon \Sigma_{g,n}\to\Sigma_{g,n}
	\end{align*}
is isotopic to the identity or equivalently, if $g\circ f^{-1}:X\to Y$ is isotopic to an isometry. In an element $[X,f]$, the mapping $f$ is called a \textit{marking} on $X$. For notation, when $L$ is the zero vector we denote $\mathcal{T}_{g,n}=\mathcal{T}_{g,n}(0,\ldots,0)$ and when there are no boundary components we simply write $\mathcal{T}_{g}$ for $\mathcal{T}_{g,0}$. 

There exists a natural group action on the Teichm\"{u}ller space that acts by changing the marking. The group is called the \textit{mapping class group} $\mathrm{Mod}_{g,n}(\Sigma_{g,n})$ and is defined as the collection of orientation-preserving diffeomorphisms of $\Sigma_{g,n}$ that fix the boundary components setwise identified up to isotopy to the identity mapping. If $[\varphi]\in\mathrm{Mod}_{g,n}(\Sigma_{g,n})$ then the action on an element $[X,f]\in\mathcal{T}_{g,n}(L)$ is given by
	\begin{align*}
		[\varphi]\cdot [X,f] = [X,f\circ\varphi^{-1}].
	\end{align*}
Equivalently, $\mathrm{Mod}_{g,n}(\Sigma_{g,n})$ can be defined as the group of orientation preserving homeomorphisms fixing the boundary components, up to homotopy. This is due to the fact that on a compact surface, any homeomorphism is isotopic to a diffeomorphism, and two orientation preserving homeomorphisms are homotopic iff they are isotopic.\par 

The \textit{moduli space} $\mathcal{M}_{g,n}(L)$ is then the space obtained through identification of points in the Teichm\"{u}ller space up to the mapping class group action. That is,
	\begin{align*}
		\mathcal{M}_{g,n}(L) = \mathcal{T}_{g,n}(L)/\mathrm{Mod}_{g,n}(\Sigma_{g,n}).
	\end{align*}
As with the Teichm\"{u}ller space, we use the shorthand notation $\mathcal{M}_{g,n}=\mathcal{M}_{g,n}(0,\ldots,0)$ and $\mathcal{M}_g=\mathcal{M}_{g,0}$.\par 
As well as a group action, there is an associated symplectic form on $\mathcal{T}_{g,n}(L)$ called the \textit{Weil-Petersson form} denoted by $\omega_{g,n}$ which is invariant under the action of the mapping class group (see Goldman \cite{Gol84}). Due to this invariance, the form passes also to the moduli space and hence provides a volume form on $\mathcal{M}_{g,n}(L)$ called the \textit{Weil-Petersson volume}
	\begin{align*}
		\frac{\wedge^{3g+n-3}\omega_{g,n}}{(3g+n-3)!}.
	\end{align*}
In particular, we write 
	\begin{align*}
		V_{g,n}(L) = \int_{\mathcal{M}_{g,n}(L)} \frac{\wedge^{3g+n-3}\omega_{g,n}}{(3g+n-3)!},
	\end{align*}
for the volume of $\mathcal{M}_{g,n}(L)$ and use the shorthand notation $V_{g,n}=V_{g,n}(0,\ldots,0)$ and $V_g=V_{g,0}$. \par
Some particularly important results concerning volumes of moduli spaces that will be made use of here are from Mirzakhani \cite{Mir13} and Mirzakhani and Zograf \cite{MZ15} and we reproduce them for the convenience of the reader. The first allows one to relate the volumes $V_{g,n}(L)$ to $V_{g,n}$.
\begin{lemma}[Mirzakhani~{\cite[Equation 3.7]{Mir13}}]
\label{lem: Mirzakhani Volume Estimates}
Given any $g,n\in\N$ and $L\in\R^n_{\geq 0}$,
	\begin{align*}
		V_{g,n}(2L)\leq e^{|L|}V_{g,n},
	\end{align*}
where $|L|=L_1+\cdots +L_n$.
\end{lemma} 
The second result shows a relationship between volumes with different genus and boundary components. For $g\to\infty$, the relation is asymptotically sharp.
\begin{lemma}[Mirzakhani~{\cite[Equation 3.20]{Mir13}}]
\label{lem: Mirzakhani Volume Relation}
Given $g,n\in\N\cup\{0\}$ with $2g-2+n\geq 0$ and $0 \leq i \leq n/2$, 
	\begin{align*}
		V_{g,n}\lesssim V_{g+i,n - 2i},
	\end{align*}
where the implied constant is independent of $g, n$ and $i$.
\end{lemma}
The last volume estimate result we need provides growth estimates for moduli space volumes in the large genus limit.
\begin{thm}[Mirzakhani and Zograf~{\cite[Theorem 1.2]{MZ15}}]
\label{thm: Mirzakhani, Zograf Volume Estimates}
There exists a universal constant $C\in (0,\infty)$ such that for any given $n\geq 0$,
	\begin{align*}
		V_{g,n}=\frac{C}{\sqrt{g}}(2g-3+n)!(4\pi^2)^{2g-3+n}\left(1+O\left(\frac{1}{g}\right)\right)
	\end{align*}
as $g\to \infty$. In particular, 
	\begin{align*}
		V_g = \frac{C}{\sqrt{g}}(2g-3)!(4\pi^2)^{2g-3}\left(1+O\left(\frac{1}{g}\right)\right),
	\end{align*}
as $g\to\infty$.
\end{thm}
Notice in particular that the volume of the moduli space is finite and hence there is a probability measure on the moduli space called the \textit{Weil-Petersson probability measure}, $\P_{g,n}$. If $A\subseteq\mathcal{M}_{g,n}$ we will write
	\begin{align*}
		\mathbb{P}_{g,n}(A)= \frac{1}{V_{g,n}}\int_{\mathcal{M}_{g,n}}\1_A(X)\df X,
	\end{align*}
where we use $\df X$ as shorthand for the Weil-Petersson volume measure and $X$ for an element of the moduli space. Moreover, one can determine the expectation of a measurable function $F \colon \mathcal{M}_{g,n}\to\R$ with respect to this measure in the usual way through the expression
	\begin{align*}
		\mathbb{E}_{g,n}(F)= \frac{1}{V_{g,n}}\int_{\mathcal{M}_{g,n}} F(X)\df X.
	\end{align*}
An extremely useful result that we will use for calculating integrals of certain functions over moduli space will be \textit{Mirzakhani's integral formula}. For this, we need to introduce the notion of cutting open a surface along a system of curves. To this end, recall that in the free homotopy class of a simple closed curve on a hyperbolic surface, there exists a unique simple closed geodesic minimising length amongst all curves in the homotopy class. When we consider a simple closed curve, we will always be considering the free homotopy class or simple closed geodesic representative in this class. In the following we will consider the notion of \textit{multicurves}.

\begin{definition} If $\gamma_1,\dots,\gamma_k$ are homotopically distinct and simple closed curves, we define their \textit{multicurve} as the formal sum $\gamma = \sum_{i = 1}^k \gamma_i$ which gives a union of curves in $S_g$. 
\end{definition}

\par 

We now seek to understand how such a multicurve cuts the surface. Fix a multicurve $\gamma=\sum_{i=1}^k\gamma_i$ and denote by $\Sigma_g\setminus\gamma$ the possibly disconnected surface with $q\geq 1$ connected components and $2k$ boundary components formed by cutting $\Sigma_g$ along the $\gamma_i$. Each such curve component $\gamma_i$ in $\gamma$ provides precisely two boundary components on $\Sigma_g\setminus\gamma$. Fix an order $\Gamma = (\gamma_1,\ldots,\gamma_k)$ for the curves in $\gamma$ and suppose that $\mathbf{x}=(x_1,\ldots,x_k)\in \mathbb{R}^k_{+}$. Denote by
	\begin{align*}
		\mathcal{M}(\Sigma_{g}\setminus\gamma, \ell(\Gamma)=\mathbf{x}),
	\end{align*}
the moduli space of hyperbolic surfaces homeomorphic to $\Sigma_{g}\setminus\gamma$ such that the length of $\gamma_i$ satisfies $\ell(\gamma_i)=x_i$. Moreover, set
	\begin{align*}
		V_g(\Gamma, \mathbf{x}) = \mathrm{Vol}(\mathcal{M}(\Sigma_{g}\setminus\gamma, \ell(\Gamma)=\mathbf{x})),
	\end{align*}
to be the volume of this moduli space. We may write $\Sigma_g\setminus\gamma$ as the disjoint union of its connected components so that
	\begin{align*}
		\Sigma_g\setminus\gamma = \bigsqcup_{i=1}^q \Sigma_{g_i,n_i},
	\end{align*}
with $\sum_{i=1}^q n_i = 2k$. For the moduli space above, we then have that
	\begin{align*}
		\mathcal{M}(\Sigma_{g}\setminus\gamma, \ell_\Gamma=\mathbf{x}) \cong \prod_{i=1}^q \mathcal{M}_{g_i,n_i}(x_{i_1},\ldots,x_{i_{n_i}}),
	\end{align*}
where the $i_1,\ldots,i_{n_i}$ are the indices corresponding to the multicurve components $\gamma_{i_1},\ldots,\gamma_{i_{n_i}}$ that form the boundary of the corresponding connected component. The volume is then just
	\begin{align*}
		V_g(\Gamma, \mathbf{x}) = \prod_{i=1}^q V_{g_i,n_i}(x_{i_1},\ldots,x_{i_{n_i}}).
	\end{align*}

For Mirzakhani's integral formula, one considers the integral of so-called geometric functions on the moduli space. These are defined from a multicurve such as $\gamma$ above. Indeed, let $F \colon \R_+^k\to\R_+$ be a symmetric measurable function, and define $F_\gamma \colon \mathcal{M}_g\to\R_+$ by
	\begin{align*}
		F_\gamma(X) \coloneqq \sum_{\sum_{i=1}^k\alpha_i\in\mathrm{Mod}_g(\Sigma_g)\cdot \gamma} F(\ell_X(\alpha_1),\ldots,\ell_X(\alpha_k))
	\end{align*}
where $\ell_X(\gamma_i)$ is the length of the simple closed geodesic in the free homotopy class of the image of $\gamma_i$ under the marking on $X$. Moreover, define 
	\begin{align*}
		M(\gamma)=|\{i = 1,\dots,k :  \gamma_i\ \text{separates a handle from the surface}\}|,
	\end{align*}
and 
	\begin{align}\label{eq:sym}
		\mathrm{Sym}(\gamma) \coloneqq \mathrm{Stab}(\gamma)/\bigcap_{i=1}^k\mathrm{Stab}(\gamma_i).
	\end{align} 
Here $\mathrm{Stab}(\gamma) = \{h\in \mathrm{Mod}_g(\Sigma_g) : h\cdot \gamma = \gamma\}$ is the stabiliser of the multicurve $\gamma$ and similarly $\mathrm{Stab}(\gamma_i) =  \{h\in \mathrm{Mod}_g(\Sigma_g) : h\cdot \gamma_i = \gamma_i\}$  of the single curve $\gamma_i$, $i = 1,\dots,k$, under the action of the mapping class group $\mathrm{Mod}_g(\Sigma_g)$. Mirzakhani's integral formula is then stated as follows.
\begin{thm}[{Mirzakhani~\cite[Theorem 7.1]{Mir07}}]
\label{thm: Mirzakhani Integral Formula}
For any multicurve $\gamma = \sum_{i = 1}^k \gamma_i$ and a symmetric measurable function $F \colon \R^k_+\to\R_+$, one has
	\begin{align*}
		\int_{\mathcal{M}_g}F_\gamma(X)\df X = \frac{1}{2^{M(\gamma)}|\mathrm{Sym}(\gamma)|}\int_{\R^k_+} F(\mathbf{x})V_{g}(\Gamma,\mathbf{x})\mathbf{x} \cdot d\mathbf{x}
	\end{align*}
	where $\mathbf{x} \cdot d\mathbf{x} = x_1\cdots x_k dx_1 \wedge \dots \wedge d x_k$ and $\Gamma = (\gamma_1,\dots,\gamma_k)$.
\end{thm}

\section{Short Geodesic Loops on Random Surfaces}\label{s:randomproof}
\label{Random Surfaces}
Consider a compact hyperbolic surface $X=\Gamma\backslash\H$ with a fundamental domain $D$. Recall that the assumption on the surfaces we considered in \eqref{eqn: ShortGeodesicBound} was the existence of an $R\geq 0$ such that for all $z,w\in D$
	\begin{align*}
		|\{\gamma\in\Gamma : d(z,\gamma w)<r\}| \leq C_0(\delta)C(X) e^{\delta r},\ \ \text{for all $\delta>0$ and $r\leq R$.}
	\end{align*}
In this section, we demonstrate that one can take $R=c\log(g)$ for sufficiently small $c$ independent of the genus $g$ if the surface $X$ has injectivity radius bounded below by $g^{-b}$ for some $b>0$ to be determined (also independently of $g$) such that the assumption is satisfied with probability tending to one as $g\to\infty$. For this, we show the sufficient condition that about any point $z\in D$ there is at most one primitive geodesic loop on the surface based at $z$ with length at most $c\log(g)$ with probability tending to one as $g\to\infty$; that is, we show Theorem \ref{t:proba}. An outline of the proof of this result is given as follows.

\subsection{Outline of the proof}
Theorem \ref{t:proba} is proven using contradiction via the following methodology. The general idea is close to the methods used by Mirzakhani and Petri \cite{MP17}: we want to deduce from the presence of two distinct geodesic loops at one point the existence of a separating multicurve, and show that the probability for such a multicurve to exist in the large genus limit tends to 0. An important difference with \cite{MP17} is that we deal here with geodesic loops instead of closed geodesics, which most notably behave differently in terms of self-intersections (Lemma \ref{lem: loop intersections}). Our main contribution is then a generalised volume product formula (Lemma \ref{lem: Sum-Volume Product}) based on finer estimates of volumes of moduli spaces from Mirzakhani and Zograf \cite{MZ15}. The dependence of all the constants on the genus renders the analysis considerably more involved, and we can also highlight in particular that in the proof of Theorem \ref{thm: Separating Multicurve Almost Surely Does Not Exist} we need additional steps to reduce our analysis to a certain topological type of multicurves that we call \emph{minimally separating}.

 We now give a detailed outline.
\begin{enumerate}
		\item Given two primitive geodesics loops of length at most $c\log(g)$ on a hyperbolic surface $X$ of genus $g$ passing through the same point, we determine an upper bound on the number of self-intersections that these two curves can have between one another and with themselves. (Lemma \ref{lem: loop intersections})
		\item The bound determined then provides an upper bound on the number of components in a multicurve obtained by taking a regular neighbourhood of the original curves and we show that for large enough genus $g$, this multicurve is separating and has total length at most $4c\log(g)$. (Lemmas \ref{lem: filling} and \ref{lem: Number Of Intersections Between Curves})
		\item We next prove an estimate on the order of growth on the sum of the products of the volumes of moduli spaces obtained by cutting along a multicurve as above, over all possible configurations of subsurface genera that such a multicurve could cut into given the number of components in the curve. (Lemma \ref{lem: Sum-Volume Product})
		\item Using this estimate, we show that asymptotically as $g\to\infty$ such a multicurve does not exist on the surface with probability tending to one as $g\to\infty$. This is done by computing an upper bound on the expected number of separating multicurves with a bounded number of components (computed by (2)) with length at most $4c\log(g)$ that can exist on a surface and showing it asymptotically tends to zero as $g\to\infty$. (Theorem \ref{thm: Separating Multicurve Almost Surely Does Not Exist})
		\item One can then conclude that with probability tending to $1$ as $g\to +\infty$, given a surface of genus $g$ there is at most one primitive geodesic loop of length at most $c\log(g)$.
	\end{enumerate}
		
\subsection{Bounds on $N_{c\log(g)}(X)$ are sufficient for condition \eqref{eqn: ShortGeodesicBound}}
Before starting the proof of Theorem \ref{t:proba}, we provide a simple argument to demonstrate that a surface for which $N_{r}(X)\leq n$ satisfies condition \eqref{eqn: ShortGeodesicBound} with the implied constant $C(X)$ dependent upon the injectivity radius of the surface. Recall that a geodesic loop based at a point is primitive if it is generated by a primitive element of the group $\Gamma$.
\begin{lemma}
\label{lem: Sufficient Assumption}
Suppose that $X=\Gamma/\H$ is a compact hyperbolic surface for which there exists an $R>0$ such that $N_R(X)\leq n$. Then for each $z,w\in D$,
	\begin{align*}
		\left|\left\{\gamma\in\Gamma : d(z,\gamma w)\leq \frac{r}{2}\right\}\right|\leq \frac{2nr}{\mathrm{InjRad}(X)}+2,\ \text{for all}\ r\leq R.
	\end{align*}
\end{lemma}
\begin{proof}
Suppose that $N_r(X)\leq n$; we will first count the number of non-identity $\gamma\in\Gamma$ for which $d(z,\gamma z)\leq r$ for a given $z\in\H$. By definition, each primitive element $\gamma\in\Gamma$ that satisfies $d(z,\gamma z)\leq r$ produces a primitive geodesic loop based at $z$ on the surface of length at most $r$. From the assumption that $N_r(X)\leq n$, this means there can be at most $n$ primitive elements that satisfy this distance bound. \par 
Given such a primitive $\gamma$, the powers $\gamma^i$ will also satisfy the distance bound $d(z,\gamma^i z)\leq r$ if they generate short enough geodesic loops. As $z$ does not necessarily lie on the axis of $\gamma$, the geodesic loop arising from the projection of the geodesic between $z$ and $\gamma^i z$ onto the surface may have length shorter than $i$ times the distance $d(z,\gamma z)$. By definition however, it has length at least the injectivity radius of the surface and thus at most the powers 
	\begin{align*}
		\gamma^{\pm i}, \quad \text{for } i=1,\ldots,\left\lfloor \frac{r}{\mathrm{InjRad}(X)}\right\rfloor,
	\end{align*}
can also satisfy $d(z,\gamma^i z)\leq r$. Such powers will however account for all of the possible group elements with $d(z,\gamma z)\leq r$. Indeed, if $\gamma'$ is an element of $\Gamma$ with $d(z,\gamma'z)\leq r$ then it is either the identity or a power of a primitive element, say $\gamma$. This follows from the fact that the surface is compact, and so all elements are hyperbolic and powers of primitives \cite[Lemma 5.4]{Ber16}. In this latter case, we necessarily have that
	\begin{align*}
		d(z,\gamma z)\leq d(z,\gamma' z).
	\end{align*} 
To see this, first suppose that $\gamma$ is a dilation of the form $\gamma \colon z\mapsto az$ for some $a>0, a\neq 1$. Then, $d(z,\gamma z)\leq d(z,\gamma^m z)$ for all $z\in\mathbb{H}$ and $m\in\mathbb{Z}$ by the monotone increasing property of the $\cosh$ function and the explicit formula for the hyperbolic distance between two points given by
	\begin{align*}
		\cosh(d(z,w))= 1+\frac{(\mathrm{Re}(z)-\mathrm{Re}(w))^2+(\mathrm{Im}(z)-\mathrm{Im}(w))^2}{2\mathrm{Im}(z)\mathrm{Im}(w)},
		\end{align*}
to compare both sides of the inequality. For a general $\gamma\in\Gamma$, we can use the fact that since $\gamma$ is hyperbolic, there is some $g\in\mathrm{PSL}(2,\R)$ for which $\gamma = g\delta g^{-1}$ where $\delta$ is a dilation as above. Then, 
	\begin{align*}
		d(z,\gamma z) = d(g^{-1}z,\delta g^{-1}z) \leq d(g^{-1}z,\delta^mg^{-1}z) = d(z,\gamma^m z),
	\end{align*}
for any $m\in\mathbb{Z}$, proving the desired inequality. Thus, this means that if $d(z,\gamma'z)\leq r$, then we also have $d(z,\gamma z)\leq r$ for $\gamma$ the primitive of $\gamma'$, and hence $\gamma'$ is accounted for as a power of one of the primitive elements that satisfy the distance inequality. With this in mind, we obtain the bound
	\begin{align*}
		|\{\gamma\in\Gamma\setminus\{\mathrm{id}\} : d(z,\gamma z)\leq r\}|\leq 2n\left\lfloor\frac{r}{\mathrm{InjRad}(X)}\right\rfloor.
	\end{align*}
For fixed $z,w\in\H$ we now count the number of $\gamma\in\Gamma$ with $d(z,\gamma w)\leq \frac{r}{2}$. Suppose there were at least $m=2n\lfloor\frac{r}{\mathrm{InjRad}(X)}\rfloor+2$ distinct non-identity elements with this property labelled $\gamma_1,\ldots,\gamma_m$. The elements $\gamma_j\gamma_1^{-1}$ are then distinct non-identity elements in $\Gamma$ for $j=2,\ldots,m$. Moreover, for each $j$ we have 
	\begin{align*}
		d(\gamma_1w, (\gamma_j\gamma_1^{-1})(\gamma_1w))\leq d(\gamma_1w,z)+d(\gamma_jw,z)<r.
	\end{align*}
This means that we have found $m-1$ distinct, non-identity elements in $\Gamma$ for which $d(\gamma_1w,\gamma(\gamma_1w))\leq r$ which is a contradiction to the above counting argument. This means that there can be at most $m-1$ such elements, and so including the identity we obtain
	\begin{align*}
		\left|\left\{\gamma\in\Gamma : d(z,\gamma w)\leq \frac{r}{2}\right\}\right|\leq \frac{2nr}{\mathrm{InjRad}(X)}+2.
	\end{align*}
\end{proof}		

\subsection{Geometry of loops and extracting a separating multicurve}

Suppose now that $N_r(X)>1$. Then, there exists some $z\in X$ that has at least two primitive geodesic loops passing through it of length at most $r$. With $r\leq c\log(g)$ for some $c>0$ to be determined, we next demonstrate that two such loops give rise to a certain separating multicurve on the surface $X$ for large enough genus $g$. This result requires an improvement on the technique of Mirzakhani and Petri \cite[Proposition 4.5]{MP17} to allow for the curve lengths to have some dependence on the genus $g$ and for the curves themselves to be geodesic loops rather than closed geodesics. We will first require the following lemma to determine the number of intersections between two such loops that have a finite number of intersections. Note the case where the two have an infinite number of intersections happens only when one loop is a subloop of the other and we will use primitivity of the loops to deal with this later in Lemma \ref{lem: Number Of Intersections Between Curves}.

\begin{lemma}
\label{lem: loop intersections}
Suppose that $\alpha$ and $\beta$ are geodesic loops of lengths $\ell(\alpha)$ and $\ell(\beta)$ respectively which have a finite number of intersections between them. Then,
	\begin{align*}
		i(\alpha,\beta) \leq \left\lceil \frac{2\ell(\alpha)}{\mathrm{InjRad}(X)}\right\rceil\left\lceil \frac{2\ell(\beta)}{\mathrm{InjRad}(X)}\right\rceil,
	\end{align*}
where $i(\alpha,\beta)=\#(\alpha\cap\beta)$ denotes the number of intersections between the two curves.
\end{lemma}

\begin{proof}
Consider a geodesic segment $\bar{\alpha}$ of $\alpha$ of length $\rho=\frac{1}{2}\mathrm{InjRad}(X)$ which has the maximal number of intersections with $\beta$ amongst all geodesic segments of $\alpha$ of length $\rho$. In this way, one obtains the upper bound
	\begin{align*}
		i(\alpha,\beta) \leq \left\lceil \frac{\ell(\alpha)}{\rho}\right\rceil i(\bar{\alpha},\beta).
	\end{align*}
Similarly, dividing $\beta$ into geodesic segments of length $\rho$, we can bound this latter intersection by the number of such segments multiplied by the intersection number between $\bar{\alpha}$ and the segment of $\beta$ with the most intersections with $\bar{\alpha}$, say $\bar{\beta}$ so that
	\begin{align*}
		i(\alpha,\beta) \leq \left\lceil \frac{\ell(\alpha)}{\rho}\right\rceil \left\lceil \frac{\ell(\beta)}{\rho}\right\rceil i(\bar{\alpha},\bar{\beta}).
	\end{align*}
Suppose that $\bar{\alpha}$ and $\bar{\beta}$ intersect at some point $p$. Then, by construction, both of these geodesic segments lie in $B_\rho(p)$. This ball however is an embedded ball in the surface by definition of $\rho$, and so they cannot intersect at another point in the ball (otherwise we would have distinct geodesics in the plane intersecting in more than one place). This gives $i(\bar{\alpha},\bar{\beta})\leq 1$ and the result follows.
\end{proof}

Notice that the previous result can be easily modified to show that the same bounds hold on the number of self-intersections of a single loop with multiplicity. By multiplicity, we mean that if the loop intersects itself in the same point multiple times then we count each of these occurrences individually. For example, if at a self-intersection point there are 6 emanating curve segments then it will mean the curve has crossed through that point three times and hence intersected itself twice so this will be counted as two intersections.  In summary this means that the total number of intersections between two curves of length at most $c\log(g)$ for some $c>0$ and themselves is 
$$O((c\log(g))^2\mathrm{InjRad}(X)^{-2}).$$ 
Recall that we assumed that $\mathrm{InjRad}(X)\geq g^{-b}$ for some $b>0$ to be chosen later (independently of $g$). With this condition, the number of intersections will be $O(c^2g^{2b}(\log(g))^2)$.\par
In constructing our desired multicurve from the geodesic loops, we will be taking a regular neighbourhood and thus need to be able to deduce properties about the resulting subsurface that is bounded by the components of the neighbourhood. For this, we will need the following result, that is an adaptation to surfaces with boundaries and non-simple curves of \cite[Lemma 2.1]{AH13}. We will say that two curves $\alpha, \beta$ on a surface $\Sigma_{g,n}$ of genus $g$ with $n$ boundaries are \emph{filling} if $\Sigma_{g,n} \setminus (\alpha \cup \beta)$ is a disjoint union of topological disks and annuli, such that each annulus is homotopic to a boundary component of $\Sigma_{g,n}$.

\begin{lemma}
\label{lem: filling}
Suppose that $\alpha$ and $\beta$ are two curves that fill $\Sigma_{g,n}$ whose intersection with one another and themselves (if there are any) are transversal, then
$$ i(\alpha,\beta)+i(\alpha,\alpha)+i(\beta,\beta) \geq 2g + n - 2,$$
where $i(\alpha_1, \alpha_2)$ is the number of intersections of the curves $\alpha_1$ and $\alpha_2$; recall as above that when $\alpha_1=\alpha_2$ then this is counted with multiplicity.
\end{lemma}
\begin{proof}
Because $\alpha$ and $\beta$ are filling, $\Sigma_{g,n} \setminus (\alpha \cup \beta)$ is a disjoint union of  topological disks and annuli. We can form the $1$-skeleton of a cellular decomposition of $\Sigma_{g,n}$ by considering the graph $\mathcal{G}(\alpha,\beta)$ whose vertex set is the set of intersection points both between $\alpha$ and $\beta$, and amongst themselves, and then adjoin to $\mathcal{G}(\alpha,\beta)$ one additional vertex and two additional edges per annuli. Let $i_0(\alpha,\beta)$ denote the number of intersection points between $\alpha$ and $\beta$ and the curves themselves counted without multiplicity, or in other words, the number of vertices in $\mathcal{G}(\alpha,\beta)$.\par 
	In the graph of $\mathcal{G}(\alpha,\beta)$ adjoined with the extra components, there will be $i_0(\alpha, \beta) + n$ vertices and $$2n+\frac{1}{2}\sum_{v\in\mathcal{G}(\alpha,\beta)} \mathrm{deg}_{\mathcal{G}(\alpha,\beta)}(v)$$ edges where $\mathrm{deg}_{\mathcal{G}(\alpha,\beta)}(v)$ denotes the degree of the vertex $v$ in the graph $\mathcal{G}(\alpha,\beta)$. Now, the sum of these degrees in $\mathcal{G}(\alpha,\beta)$ will be
		\begin{align*}
			4i_0(\alpha,\beta) + 2( i(\alpha,\beta)+i(\alpha,\alpha)+i(\beta,\beta) - i_0(\alpha,\beta) ) = 2(i(\alpha,\beta)+i(\alpha,\alpha)+i(\beta,\beta) + i_0(\alpha,\beta)).
		\end{align*}
	To see this, notice that each vertex appearing in the graph $\mathcal{G}(\alpha,\beta)$ will have degree 4 plus an extra 2 edges will emanate from a vertex for every additional crossing of one of the curves at that point. Hence the first term on the left hand side accounts for the base degree of 4 at each vertex and the second term accounts for the total number of additional crossings at all vertices in $\mathcal{G}(\alpha,\beta)$. This total number of additional crossings will be precisely the total number of crossings which is 
		\begin{align*}
			i(\alpha,\beta)+i(\alpha,\alpha)+i(\beta,\beta),
		\end{align*}
	minus the number of crossings that are the initial crossings of the curves (with each other or themselves) which is $i_0(\alpha,\beta)$. This means, that the total number of edges in this adjoined graph will be
		\begin{align*}
			2n+i(\alpha,\beta)+i(\alpha,\alpha)+i(\beta,\beta) + i_0(\alpha,\beta).
		\end{align*}
	Let $D$ be the number of $2$-cells in this cellular decomposition of the surface so that $D\geq n$. \par
	Then the Euler characteristic of $\Sigma_{g,n}$ is
	\begin{align*}
		\chi(S_{g,n}) = 2 - 2g - n &= i_0(\alpha,\beta) + n - (2n + i(\alpha,\beta)+i(\alpha,\alpha)+i(\beta,\beta) + i_0(\alpha,\beta)) + D\\
		&\geq -(i(\alpha,\beta)+i(\alpha,\alpha)+i(\beta,\beta))
	\end{align*}
By rearranging, we obtain $$ i(\alpha,\beta)+i(\alpha,\alpha)+i(\beta,\beta) \geq 2g - 2 + n.$$
\end{proof}

We can now show that two geodesic loops based at the same point imply the existence of a separating multicurve for large enough $g$.

\begin{lemma}
\label{lem: Number Of Intersections Between Curves}
Suppose that $\alpha$ and $\beta$ are primitive geodesic loops in the surface $X=\Sigma_g$ based at the same point with lengths bounded by $c\log(g)$ for some constant $c>0$. Moreover, assume that for some $0<b<\frac{1}{2}$, $\mathrm{InjRad}(X)>g^{-b}$. Then, there exists a separating multicurve $\gamma$ on $\Sigma_g$ consisting of $O(c^2g^{2b}(\log(g))^2)$ simple closed geodesics whose total length is bounded by $4c\log(g)$ for $g$ sufficiently large.
\end{lemma}
\begin{proof}
Given the two curves $\alpha$ and $\beta$, there are two possibilities. Firstly, the two loops will have a finite number of intersections between them (at least one since they intersect at the base point of the loop). In this case we have a bound on the total number of intersections both between the two curves and their self-intersections by Lemma \ref{lem: loop intersections} of order $O(c^2g^{2b}(\log(g))^2)$ using the condition on the injectivity radius. The second possibility is that one of the loops is a subloop of the other. In this case, we consider just the longer of the two curves. Again this curve has at least one self-intersection since for it to be distinct from the other curve it must contain more than one subloop. The total number of self-intersections of this curve is also again of order $O(c^2g^{2b}(\log(g))^2)$ using Lemma \ref{lem: loop intersections}. \par 
Consider a regular neighbourhood of the curves in either possibility described above in $\Sigma_g$. The boundary of this neighbourhood will be a collection of disjoint simple closed curves and we consider the multicurve $\gamma$ consisting of the simple closed geodesics that are freely homotopic to the boundary curves (discarding any such repeated curves). By construction when taking the neighbourhood of the set, each boundary component will be homotopic to simple closed segments of $\alpha\cup\beta$ (or just one of the curves in the second case) with each such segment appearing exactly twice (the portion of the neighbourhood either side of the union of the curves). Since the geodesics in the free homotopy classes are length minimising, their total sum must then be at most twice the total sum of the curves $\alpha$ and $\beta$ from this double counting and so the total length of the multicurve constructed is bounded by $4c\log(g)$. \par 

\begin{figure}[ht!]
\includegraphics[scale=0.2]{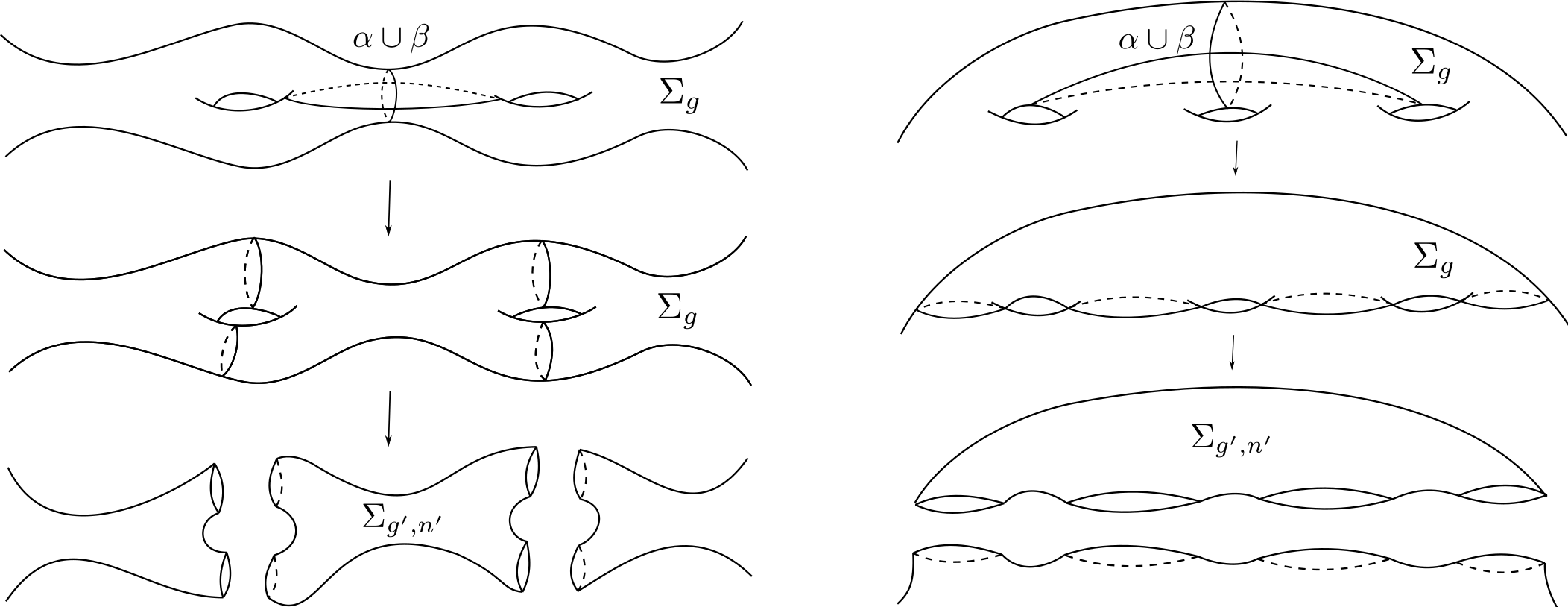}
\caption{Some possibilities of the formation of the subsurface $\Sigma_{g',n'}$ from the regular neighbourhood of $\alpha \cup \beta$ in $\Sigma_g$. One begins by taking the regular neighbourhood of the union $\alpha \cup \beta$ and homotoping the boundary components to their geodesic representations. Then one cuts along these geodesics to obtain the subsurface $\Sigma_{g',n'}$.}
\label{fig:multicurve}
\end{figure}

If one considers the graph whose vertices are the points of intersection of the curve(s) and edges being the geodesic segments between the curves, then one may homotope this graph to a wedge of circles. In each possibility, the primitivity of the curves ensures that we have at least two distinct circles in this wedge and so the regular neighbourhood bounds a non-trivial hyperbolic surface. It is clear by construction that $\alpha$ and $\beta$ are together filling curves for the subsurface constructed by this regular neighbourhood as it is non-trivial. Thus, if $(g',n')$ is the signature of this subsurface $\Sigma_{g',n'}$ we have by Lemma \ref{lem: filling} that
	\begin{align*}
		2g'+n'-2 \leq I,
	\end{align*}
where $I=i(\alpha,\beta)+i(\alpha,\alpha)+i(\beta,\beta)$. By Lemma \ref{lem: loop intersections} applied to each of the intersections, we obtain that $I=O(c^2g^{2b}(\log(g))^2)$ in either case by hypothesis on the surface. \par 
If also $\alpha$ and $\beta$ filled the surface $\Sigma_g$ then by the same argument one would have that
	\begin{align*}
		2g-2 \leq I,
	\end{align*}
which for $g$ sufficiently large is not possible since $I=o(g)$ as $b<\frac{1}{2}$, and so $\gamma$ must be separating when $g$ is large enough. Two possibilities of how this multicurve could separate the surface are given in Figure \ref{fig:multicurve}. The number of components in $\gamma$ is given by $n'$ which from the inequality $n'\leq 2g'+n' \leq I+2$ is seen to be of order $O(c^2g^{2b}(\log(g))^2)$ as required.
\end{proof}

So with this result, from two primitive geodesic loops based at a point of length at most $c\log(g)$, we obtain for large enough genus a separating multicurve with total length at most $4c\log(g)$ consisting of disjoint simple closed geodesics. We will next investigate how such a multicurve can be realised on a surface and show that with probability tending to one as $g\to\infty$, such a multicurve can not exist on a random surface $X$ with injectivity radius bound given previously. This will mean that $N_r(X)\leq 1$ for all $r\leq c\log(g)$ for some $c>0$ and that we have a suitable bound on the number of group elements desired, both with high probability.

\subsection{Proving Theorems \ref{thm:MainRandom} and \ref{t:proba}}
We require an estimate on the product of volumes of moduli spaces of the subsurfaces obtained from cutting along the multicurve when the lengths of the curves can depend on the genus. In particular, we wish to see how the sum of such products can grow over all possible genera configurations on the subsurfaces with a given number of boundary components on each subsurface (in fact we will only require a special case of this for when the number of subsurfaces is 2, but we include the more general result here as it is of interest in its own right). The starting point for this is the relation between different volumes given in Mirzakhani \cite[Lemma 3.2]{Mir13} which has been reproduced here in Lemma \ref{lem: Mirzakhani Volume Relation} and the growth estimate on volumes of moduli spaces from Mirzakhani and Zograf \cite{MZ15} stated in Theorem \ref{thm: Mirzakhani, Zograf Volume Estimates}.

\begin{lemma}
\label{lem: Sum-Volume Product}
Suppose that $q,k(g),n_1(g),\ldots,n_q(g)\in\N$ with $2\leq q\leq k(g)+1$, $\sum_{i=1}^q n_i(g) = 2k(g)$ and $k(g)=O(g^d)$ for some $0<d<1$, then
	\begin{align*}
		\sum_{\{g_i\}} \prod_{i=1}^q V_{g_i,n_i(g)}= O\left(\frac{V_g D^{k(g)} \sqrt{k(g)}}{g^{\frac{1}{2}(q-1)}}\right),
	\end{align*}
as $g\to\infty$ where the sum is over all ordered sets of $\{g_i\}_{i=1}^q\subseteq\mathbb{Z}_{\geq 0}$ satisfying $\sum_{i=1}^q g_i= g+q-k(g)-1$ and $2g_i-3+n_i\geq 0$ for all $i=1,\ldots,q$ and $D$ is some universal constant independent of all the parameters. 
\end{lemma}
\begin{proof}
By Lemma \ref{lem: Mirzakhani Volume Relation}, one has 
	\begin{align*}
		V_{g_i,n_i(g)}\lesssim \begin{cases}
								V_{g_i+ n_i(g)/2,0} & \text{for $n_i(g)$ even,}\\
								V_{g_i+n_i(g)/2 -1/2,1} & \text{for $n_i(g)$ odd.}
							\end{cases}
	\end{align*}
In either case, by Theorem \ref{thm: Mirzakhani, Zograf Volume Estimates}
	\begin{align*}
		V_{g_i,n_i(g)} \lesssim \frac{C(2g_i+n_i(g)-3)!(4\pi^2)^{2g_i+n_i(g)-3}}{\max\{1,\sqrt{g_i+n_i(g)/2 -1}\}}\left(1+O\left(\frac{1}{g_i+n_i(g)/2}\right)\right).
	\end{align*}
This latter remainder term can be bounded by some $C'$ independent of $g_i$ and $n_i(g)$ and so we have
	\begin{align*}
		\frac{1}{V_g}\sum_{\{g_i\}} \prod_{i=1}^q V_{g_i,n_i(g)} \lesssim D^{k(g)}\sum_{\{g_i\}}\frac{\prod_{i=1}^q\frac{1}{\max\{1,\sqrt{g_i}\}}(2g_i-3+n_i(g))!(4\pi^2)^{2g_i-3+n_i(g)}}{\frac{1}{\sqrt{g}}(2g-3)!(4\pi^2)^{2g-3}},
	\end{align*}
for some constant $D$ independent of the $n_i, g, k$ and $q$. 	
To tackle the factorial terms, we use Stirling's approximation to infer that $n!\asymp \sqrt{n}\left(\frac{n}{e}\right)^n$ so that the summand is bounded up to a constant uniform in $g$, $q$, the $n_i(g)$ and $k(g)$ by
	\begin{align*}
		\frac{\sqrt{g}\prod_{i=1}^q (2g_i-3+n_i(g))^{2g_i-\frac{5}{2}+n_i(g)}\left(\frac{4\pi^2}{e}\right)^{2g_i-3+n_i(g)}}{\left(\frac{4\pi^2}{e}\right)^{2g-3}(2g-3)^{2g-\frac{5}{2}}\prod_{i=1}^q\max\{1,\sqrt{g_i}\}}.
	\end{align*}
Notice that 
	\begin{align*}
		\left(\frac{4\pi^2}{e}\right)^{\sum_{i=1}^q (2g_i-3+n_i(g))-2g+3}=\left(\frac{4\pi^2}{e}\right)^{1-q}\leq 1,
	\end{align*}
since $q\geq 2$. Next, 
	\begin{align*}
		\frac{\sqrt{g}}{\prod_{i=1}^q\max\{1,\sqrt{g_i}\}}=\frac{\sqrt{k(g)+1-q+\sum_{i=1}^q g_i}}{\prod_{i=1}^q\max\{1,\sqrt{g_i}\}} =O(\sqrt{k(g)}).
	\end{align*}
Lastly, one can observe that
	\begin{align*}
		\prod_{i=1}^q (2g_i-3+n_i(g))^{2g_i-\frac{5}{2}+n_i(g)}\leq \prod_{i=1}^q \left(2g_i-\frac{5}{2}+n_i(g)\right)^{2g_i-\frac{5}{2}+n_i(g)}.
	\end{align*}
Thus up to a constant independent of $k(g)$, $g$ and $q$ the sum of the products is bounded by
	\begin{align*}
		\sqrt{k(g)}\sum_{\{g_i\}}\frac{\prod_{i=1}^q \left(2g_i-\frac{5}{2}+n_i(g)\right)^{2g_i-\frac{5}{2}+n_i(g)}}{(2g-3)^{2g-\frac{5}{2}}}.
	\end{align*}
We now bound this summation by the number of possible ordered sets $\{g_i\}$ subject to the given Euler characteristic constraints multiplied by an upper bound on the summand itself. The former is clearly bounded above by the number of possible tuples $(g_1,\ldots,g_q)\in\mathbb{Z}_{\geq 0}^q$ that are solutions to
	\begin{align*}
		g_1+\ldots +g_q = g+q-k(g)-1,
	\end{align*}
using the first Euler characteristic constraint. However, the number of solutions to this is equal to 
	\begin{align*}
		{{g+2(q-1)-k(g)} \choose q-1} \leq (g+2(q-1)-k(g))^{q-1} \leq (2g-3)^{q-1},
	\end{align*}
for $g$ sufficiently large, with the latter inequality coming from the fact that $q\leq k(g)+1$ and $k(g) \leq g-3$ when $g$ is sufficiently large. For the maximum of the summand, we require an upper bound on the product term in the summand. Notice that by the Euler characteristic constraints, we are seeking the maximum of
	\begin{align*}
		\prod_{i=1}^q \left(2g_i-\frac{5}{2}+n_i(g)\right)^{2g_i-\frac{5}{2}+n_i(g)},
	\end{align*}
subject to 
	\begin{align*}
		\sum_{i=1}^q 2g_i-\frac{5}{2}+n_i(g) = 2g-\frac{q}{2}-2.
	\end{align*}
A product of this form attains the maximum value when all but one of the terms in the product are equal to 1 and the last component is determined by the summation condition. Hence, it is bounded above by
	\begin{align*}
		\left(2g-2-\frac{q}{2}-(q-1)\right)^{2g-2-\frac{q}{2}-(q-1)}\leq (2g-3)^{2g-\frac{3}{2}q-1},
	\end{align*}
where the latter inequality comes from the fact that $q\geq 2$. Combining these, we obtain that 
	\begin{align*}
		\frac{1}{V_g}\sum_{\{g_i\}} \prod_{i=1}^q V_{g_i,n_i(g)}= O\left(\frac{D^{k(g)}\sqrt{k(g)}(2g-3)^{2g-\frac{3}{2}q-1+(q-1)}}{(2g-3)^{2g-\frac{5}{2}}}\right) = O\left(\frac{D^{k(g)} \sqrt{k(g)}}{g^{\frac{1}{2}(q-1)}}\right).
	\end{align*}
\end{proof}

We now show that a separating multicurve as in Lemma \ref{lem: Number Of Intersections Between Curves} existing on a surface $\Sigma_g$ tends to zero in the Weil-Petersson probability asymptotically as $g\to\infty$. To this end, let $K(g)$ denote the maximal number of components in the minimally separating multicurve, so that by Lemma \ref{lem: Number Of Intersections Between Curves}, we shall look at $K(g)$ of the form $K(g)=O(c^2g^{2b}(\log(g))^2)$ for some $0<b<\frac{1}{2}$ and $c>0$ to be chosen. In fact, for the sake of simplifying the exposition of the proof we will consider $K(g)=O(g^d)$ with $d$ sufficiently small. In our case, we can take $d=2b+\varepsilon$ for any $\varepsilon>0$ if one considers $g$ large enough.\par 
Suppose that we have such a multicurve $\gamma$ then either $\gamma$ is \textit{minimally separating} in the sense that any sub-multicurve does not separate the surface or, we can find a sub-multicurve that separates the surface and trivially satisfies the same conditions on the length and number of curve components as $\gamma$. By recursively extracting sub-multicurves in this manner we will arrive at one that is minimally separating due to the fact that the number of curve components is finite, and at least one component is required to separate the surface. Thus we can show that a separating multicurve of the form we describe does not exist with probability tending to one as $g\to\infty$ by showing that minimally separating multicurve with the same length and curve component restrictions occur with probability tending to zero as $g\to\infty$. \par 
The reason that we reduce to these minimally separating multicurves is because their geometry is particularly accessible to us. Indeed, we can understand exactly how they cut a surface with the following result.

\begin{lemma}
\label{lem:MinSepMulticurveGraph}
A minimally separating multicurve $\gamma$ with $k$ components separates the surface into exactly two connected subsurfaces with $k$ boundary components each.
\end{lemma}
\begin{proof}
Consider the dual graph to the multicurve $\gamma$. This is the graph whose vertex set is the connected components of the cut surface weighted with the genus and number of boundary components of that surface, and the edge set consists of an edge between two vertices for each component of the multicurve that creates a common boundary between the two surface components represented by these vertices when one cuts the surface with that multicurve component. If this graph had more than two vertices, then one could fix any two vertices that are connected by at least one edge and consider the multicurve associated to all edges of the graph aside from the edges joining these two vertices. This multicurve by construction is a sub-multicurve of $\gamma$ and separates the surface, since it would disconnect the two fixed vertices (and hence the corresponding connected components) from the other vertices in the graph which contradicts the fact that the multicurve is minimally separating. 

There are $k$ boundary components on each of the subsurfaces. Indeed, any further boundary components on one of them (and hence less on the other since there are $2k$ boundaries in total) would originate from some of the curves in the multicurve cutting open holes on this connected component, and so such curves could be removed from the multicurve producing a smaller multicurve that still disconnects the surface -- a contradiction to the minimal separating property. In terms of the dual graph to the multicurve described above, this is equivalent to there being no self-loops at either vertex. 
\end{proof}

In addition to this property, given a minimally separating multicurve $\gamma$ we can precisely understand the symmetry group $\mathrm{Sym}(\gamma)$. In fact, in \cite[Section 4.1 and proof of Theorem 4.2 Claim 2]{Mir13}, Mirzakhani also studies multicurves of this type and states that $|\mathrm{Sym}(\gamma)|=k!$ when $\gamma$ has $k$ components. As this is a crucial point for us, we include a proof.

\begin{lemma}
\label{lem:MinSepMulticurveGroup}
If $\gamma$ is a minimally separating multicurve with $k$ components, then $\mathrm{Sym}(\gamma) \simeq \mathfrak{S}_k$, where $\mathfrak{S}_k$ is the symmetric group on $\{1, \ldots, k\}$.
\end{lemma}

\begin{proof}
Recall the definition from \eqref{eq:sym} of the symmetry group $$\mathrm{Sym}(\gamma) = \mathrm{Stab}(\gamma)/\bigcap_{i=1}^k\mathrm{Stab}(\gamma_i).$$
By construction, $\mathrm{Sym}(\gamma)$ can be identified with a subgroup $\mathfrak h$ of $\mathfrak{S}_k$. Indeed, $\mathrm{Stab}(\gamma)$ acts on the $k$ curves by permutations and the quotient makes this action faithful. To show that $\mathfrak h = \mathfrak{S}_k$ it suffices to show that the subgroup $\mathfrak{h}$ contains the transpositions.

This is the case as for any two components $\gamma_1$ and $\gamma_2$ of $\gamma$, we can find an orientation preserving homeomorphism $h$ that permutes $\gamma_1$ and $\gamma_2$ and leaves the other curves invariant. To see this, we use the existence of the following elementary transformation: Let $b_1$ and $b_2$ be two boundary components of a connected surface and $a$ an arc joining the two components. Let $\Omega$ be a regular neighbourhood of $a \cup b_1 \cup b_2$. There exists an orientation preserving homeomorphism that exchanges $b_1$ and $b_2$ and is equal to the identity outside of $\Omega$.\par 
Now to construct the homeomorphism $h$, we cut open the surface along the multicurve $\gamma$. By Lemma \ref{lem:MinSepMulticurveGraph}, we obtain two subsurfaces with $k$ boundary components, and such that each of the $k$ curves of $\gamma$ contribute one boundary component to each subsurface. Let $b_1$ and $b_2$ be the two boundary components coming respectively from $\gamma_1$ and $\gamma_2$ on one of the two subsurfaces. We can connect $b_1$ and $b_2$ by an arc $a$ that does not touch the other boundary components. We can therefore find a homeomorphism that swaps $b_1$ and $b_2$ and is the identity in a neighbourhood of all the other boundary components. Similarly, we can swap the boundaries coming from $\gamma_1$ and $\gamma_2$ on the second subsurface. Moreover, since the neighbourhoods that we consider in each of the subsurfaces when swapping the boundary components can be taken to be homeomorphic, we can arrange it so that the homeomorphisms on each subsurface match up in a homeomorphic way across the geodesics $\gamma_1$ and $\gamma_2$ when we glue back together the boundary components of each curve of $\gamma$. By this procedure we have built a homeomorphism $h$ that maps $\gamma_1$ to $\gamma_2$ and $\gamma_2$ to $\gamma_1$ while being the identity map in a neighbourhood of the other curves.
\end{proof}

\begin{thm}
\label{thm: Separating Multicurve Almost Surely Does Not Exist}
Choosing $c$ and $d$ sufficiently small independently of $g$ and $K(g)=O(g^d)$, we have that
\begin{align*}
	\P_g\left(X\in\mathcal{M}_g : \parbox{8cm}{$X$ contains a separating multicurve $\gamma$ with at most K(g) disjoint simple closed curve components of total length at most $4c\log g$.}\right)\to 0,
\end{align*}
as $g\to\infty$. In fact there exists a universal constant $\delta > 0$ such that this probability is $O(g^{-\frac12 + \delta (c+d)})$, where the implied constant is independent of $c$ and $d$.
\end{thm}
\begin{proof}
Suppose that $\gamma$ is a separating multicurve satisfying the desired properties on its length and number of curve components. As has been outlined in the lines preceding Lemma \ref{lem:MinSepMulticurveGraph}, we can extract a minimally separating multicurve from $\gamma$ satisfying the same hypotheses on the number of components and their lengths. The probability that we are interested in computing is thus bounded by
	\begin{align*}
		\P_g\left(X\in\mathcal{M}_g : \parbox{8cm}{$X$ contains a minimally separating multicurve $\gamma$ with at most $K(g)$ disjoint simple closed curve components of total length at most $4c\log g$.}\right),
	\end{align*}
which we will now proceed to bound. Let $K(g) = O(g^d)$. Suppose that $N(X)$ is the number of minimally separating multicurves $\gamma = \sum_{i = 1}^{k(\gamma)} \gamma_i$ on $X$ for each natural number $k(\gamma) \leq K(g)$ and $\ell(\gamma)\leq 4c\log(g)$. Then the event described above is given by $\{X \in \cM_g : N(X) \geq 1\}$. By Markov's inequality, we have
$$\P_g(N(X) \geq 1) \leq \frac{1}{V_g} \int_{\cM_g} N(X) \df X.$$
Let us now bound $\int N(X) \df X$ using Mirzakhani's integration formula Theorem \ref{thm: Mirzakhani Integral Formula} and the volume estimates for Weil-Petersson volume.

 Fix $1 \leq k \leq K(g)$ and consider minimally separating $\gamma$ with $k(\gamma) = k$. 
Define the following non-negative symmetric function
$$F(x_1,\dots,x_{k}) = \1_{[0,L]} (x_1+\dots+x_{k}), \quad (x_1,\dots,x_k) \in \R^{k}_+,$$
where $L=4c\log(g)$. Given a multicurve $\gamma = \sum_{i = 1}^{k} \gamma_i$ on $\Sigma_g$ as above, the associated geometric function is obtained by
$$F_\gamma(X) = \sum_{\gamma'\in [\gamma]}\1_{[0,L]}(\ell_{\gamma'_1}(X)+\cdots + \ell_{\gamma'_{k}}(X)),$$
where the sum runs over all multicurves $\gamma' = \sum_{i = 1}^{k} \gamma_i'$ in the mapping class group orbit of $[\gamma]$ and $\{\gamma_i' : 1 \leq i \leq k\}$ is the set of components of $\gamma'$. Then for any $X \in \cM_g$, we have
	\begin{align*}
		 N(X) \leq \sum_{k=1}^{K(g)}\sum_{\substack{[\gamma] \\ k(\gamma) = k}} F_\gamma(X),
	\end{align*}
	where this inner summation runs over all possible mapping class group orbits of minimally separating multicurves $\gamma = \sum_{i = 1}^{k} \gamma_i$ with $k(\gamma) = k$ and $k = 1,\dots,K(g)$. We thus obtain the following bound on the probability of interest:

	\begin{align*}
		\P_g(N(X) \geq 1) \leq \frac{1}{V_g} \sum_{k=1}^{K(g)}\sum_{\substack{[\gamma] \\ k(\gamma) = k}}\int\limits_{\mathcal{M}_g}F_\gamma(X)\df X.
	\end{align*}

We may pass this integral over the moduli space to an integral over Euclidean space via the Mirzakhani integral formula provided in Theorem \ref{thm: Mirzakhani Integral Formula}, and hence obtain an upper bound to the above of the form
	\begin{align*}
		\frac{1}{V_g}\sum_{k=1}^{K(g)}\sum_{\substack{[\gamma] \\ k(\gamma)=k}} \frac{1}{|\mathrm{Sym}(\gamma)| 2^{M(\gamma)}} \int_{\R^{k}_{\geq 0}} \1_{[0,L]}(x_1+\cdots +x_{k})x_1\cdots x_{k} V_g(\Gamma, \mathbf{x}) \bigwedge_{i=1}^{k}\df x_i, 
	\end{align*}
	where $V_g(\Gamma, \mathbf{x})$, $\mathrm{Sym}(\gamma)$ and $M(\gamma)$ are defined as in the lines preceding Theorem \ref{thm: Mirzakhani Integral Formula} for $\mathbf{x}=(x_1,\ldots,x_k) \in \R_+^{k}$ and $\Gamma = (\gamma_1,\dots,\gamma_k)$.
	
	Let us now proceed to estimate the above quantity. By Lemma \ref{lem:MinSepMulticurveGroup}, we know that the minimal separating property of $\gamma$ means that $|\mathrm{Sym}(\gamma)| = k!$.

Next, let us address the volume term $V_g(\Gamma,\mathbf{x})$. 
Using Lemma \ref{lem:MinSepMulticurveGraph}, the cut surface necessarily has only two components with $k$ boundaries on each component. Thus, this volume term is of the form
	\begin{align*}
		V_g(\Gamma, \mathbf{x}) = V_{g_1(\gamma),k}(\mathbf{x}^1)V_{g_2(\gamma),k}(\mathbf{x}^2),
	\end{align*}
where $g_1(\gamma)+g_2(\gamma)+k-1=g$ and the length vectors $\mathbf{x}^1$ and $\mathbf{x}^2$ are given by the coordinates of $\mathbf{x}$ that correspond to which components of the multicurve form the boundaries of the two subsurfaces. These volumes can be bounded using the volume estimates of Lemma \ref{lem: Mirzakhani Volume Estimates} to obtain
	\begin{align*}
		V_g(\Gamma, \mathbf{x}) \leq e^{x_1+\cdots+x_k}V_{g_1(\gamma),k}V_{g_2(\gamma),k}.
	\end{align*}

Thus we obtain that
\begingroup
	\allowdisplaybreaks
	\begin{align*}
		&		\frac{1}{V_g}\sum_{k=1}^{K(g)}\sum_{\substack{[\gamma] \\ k(\gamma)=k}} \frac{1}{|\mathrm{Sym}(\gamma)| 2^{M(\gamma)}} \int_{\R^k_{\geq 0}} \1_{[0,L]}(x_1+\cdots +x_k)x_1\cdots x_k V_g(\Gamma, \mathbf{x}) \bigwedge_{i=1}^k\df x_i \\
		&\leq \frac{1}{V_g}\sum_{k=1}^{K(g)}\sum_{\substack{[\gamma]\\ k(\gamma)=k}} \frac{1}{k!}V_{g_1(\gamma),k}V_{g_2(\gamma),k} \int_{\R^k_{\geq 0}} \1_{[0,L]}(x_1+\cdots +x_k)e^{x_1+\cdots +x_k}x_1\cdots x_k \bigwedge_{i=1}^k\df x_i \\
		&\leq \frac{1}{V_g}\sum_{k=1}^{K(g)}\sum_{\substack{[\gamma]\\ k(\gamma)=k}}\frac{1}{k!} V_{g_1(\gamma),k}V_{g_2(\gamma),k}e^L \int_{\R^k_{\geq 0}} \1_{[0,L]}(x_1+\cdots +x_k)x_1\cdots x_k \bigwedge_{i=1}^k\df x_i \\
		&\leq \frac{1}{V_g}\sum_{k=1}^{K(g)}\sum_{\substack{[\gamma]\\ k(\gamma)=k}}\frac{1}{k!} V_{g_1(\gamma),k}V_{g_2(\gamma),k}e^LL^{2k}k^{-k}, \\
	\end{align*}
\endgroup
with the factor $L^{2k}k^{-k}$ arising from the fact that the maximum of $x_1\cdots x_k$ subject to $\sum_{i=1}^k x_i= L$ arises when each $x_i$ is equal to $Lk^{-1}$ and the measure of the set $\sum_{i=1}^kx_i=L$ is bounded by $L^k$.\par 
Notice that the sum over $[\gamma]$ with $k(\gamma)=k$ may be re-written as the sum over the pairs $(g_1,g_2)\in\mathbb{Z}_{\geq 0}^2$ satisfying the Euler characteristic criteria $g_1+g_2+k-1=g$, multiplied by the number of mapping class group orbits of the minimally separating multicurves with $k$ components that separate the surface into a topological decomposition of a genus $g_1$ and genus $g_2$ subsurface each with $k$ boundary components. Since two such multicurves are in the same mapping class orbit if and only if their dual multicurve graphs described above are the same, the number of these mapping class group orbits is precisely equal to the number of possible different dual multicurve graphs that can arise from our multicurves of $k$ components. However, as described above each such graph has two vertices, k edges and no self-loops and so there is precisely one orbit for each genus $(g_1,g_2)$ decomposition. Hence, the probability we are interested in is bounded by
	\begin{align*}
		\frac{1}{V_g}\sum_{k=1}^{K(g)}\frac{e^LL^{2k}}{k!k^k}\sum_{\{(g_1,g_2): g_1+g_2+k-1=g\}} V_{g_1,k}V_{g_2,k}.
	\end{align*}
We now make use of Lemma \ref{lem: Sum-Volume Product} for $q=2$, to deduce that
	\begin{align*}
		\sum_{\{(g_1,g_2): g_1+g_2+k-1=g\}} V_{g_1,k}V_{g_2,k} = O\left(\frac{V_g D^k\sqrt{k}}{g^{\frac{1}{2}}}\right),
	\end{align*}
with the implied constant being independent of $g$ and $k$. Up to a constant, we thus have the probability bounded by
	\begin{align*}
		g^{-\frac{1}{2}}\sum_{k=1}^{K(g)}\frac{e^LL^{2k}D^k\sqrt{k}}{k!k^k}.
	\end{align*}
Using $L=4c\log(g)$ we can write $e^L=g^{4c}$ and bound $\sqrt{k}\leq \sqrt{K(g)}\lesssim g^{\frac{1}{2}d}$. Moreover, using Stirling's approximation, $k!$ is bounded up to a constant above and below by $k^{k+\frac{1}{2}}e^{-k}$. In particular, up to bounding by a universal constant, we may replace $k^kk!$ by $(2k)!e^{k}$ providing an upper bound of the form
	\begin{align*}
		g^{-\frac{1}{2}+4c+\frac{1}{2}d}\sum_{k=1}^{K(g)}\frac{(4ce^{-\frac{1}{2}} D^{\frac12}\log(g))^{2k}}{(2k)!}\leq g^{-\frac{1}{2}+4c+\frac{1}{2}d}\cosh(4ce^{-\frac{1}{2}}D^{\frac12}\log(g))\leq g^{\delta(c+d)-\frac{1}{2}},
	\end{align*}
 for some universal constant $\delta > 0$. Taking $c$ and $d$ sufficiently small, we can then insist that this probability tends to zero as $g\to\infty$.
\end{proof} 

We now combine Theorem \ref{thm: Separating Multicurve Almost Surely Does Not Exist} with Theorem 4.2 of \cite{Mir13} to obtain Theorem \ref{t:proba}.

\begin{proof}[Proof of Theorem \ref{t:proba}]
Let $b$ and $c$ be chosen such that 
	\begin{align*}
		A_g^{b,c} = \left\{X\in\mathcal{M}_g : \parbox{10cm}{$X$ contains a separating multicurve $\gamma$ with at most $K(g)$ disjoint simple closed curve components of total length at most $4c\log g$.}\right\},
	\end{align*}
where $K(g)=O(g^d)$, $d=2b+\varepsilon$ for any $\varepsilon>0$ tends to zero as $g\to\infty$. By Theorem \ref{thm: Separating Multicurve Almost Surely Does Not Exist} when these constants are suitably chosen, the rate of this is $O(g^{-\frac{1}{2}+\delta(d+c)})$ for some universal constant $\delta$. Moreover, let 
	\begin{align*}
		B_g^b = \{X\in\mathcal{M}_g:\ \mathrm{InjRad}(X)\leq g^{-b}\}.
	\end{align*}
By a result of Mirzakhani \cite[Theorem 4.2]{Mir13}, we have that
	\begin{align*}
		\mathbb{P}_g(B_g^b) = O(g^{-2b}).
	\end{align*}
Notice then that 
	\begin{align*}
		(\mathcal{M}_g\setminus A_g^{b,c})\cap (\mathcal{M}_g\setminus B_g^b)\subseteq \{X\in(\mathcal{M}_g)_{\geq g^{-b}} : N_{c\log g}(X) \leq 1\}.
	\end{align*}
Indeed, suppose $X$ is contained in the left-hand side then by definition of $B_g^b$ it is contained in $(\mathcal{M}_g)_{\geq g^{-b}}$. Moreover, if it had $N_{c\log(g)}>1$ then Lemma \ref{lem: Number Of Intersections Between Curves} would imply that there would exist a multicurve on $X$ of the form described in the definition of $A_g^{b,c}$ which is a contradiction and thus the inclusion holds. This means that Theorem \ref{t:proba} follows from a lower bound on the probability of the event on the left-hand side. But this can be determined as follows:
	\begin{align*}
		\mathbb{P}_g((\mathcal{M}_g\setminus A_g^{b,c})\cap (\mathcal{M}_g\setminus B_g^b)) &= \mathbb{P}_g(\mathcal{M}_g\setminus (A_g^{b,c}\cup B_g^b))\\
		&\geq 1 - (\mathbb{P}_g(A_g^{b,c}) + \mathbb{P}_g(B_g^b))\\
		&= 1- O(g^{-\frac{1}{2}+\delta(b+c)}+g^{-2b}),
	\end{align*}
by applying Theorem 4.2 of \cite{Mir13} and Theorem \ref{thm: Separating Multicurve Almost Surely Does Not Exist} above as required.
\end{proof}
Thus, if one sets $\cA_g^{b,c}$ to be the event 
	\begin{align*}
		\{X\in(\mathcal{M}_g)_{\geq g^{-b}} : N_{c\log g}(X) \leq 1\},
	\end{align*}
then combining Theorems \ref{t:proba}, \ref{thm:Main} and Lemma \ref{lem: Sufficient Assumption} gives that the bounds of Theorem \ref{thm:MainRandom} hold for any surface in $\cA_g^{b,c}$ which has probability tending to $1$ as $g\to +\infty$ with the rate given by Theorem \ref{t:proba}.

\section*{Acknowledgements}

We thank Bram Petri for patiently explaining to us the details of \cite{MP17} and sharing his ideas on the proof of Theorem \ref{t:proba}. We are grateful to Ara Basmajian, Mikolaj Fraczyk, Ilmari Kangasniemi, Vadim Kulikov, Michael Magee, Laura Monk, Mark Pollicott, Peter Sarnak, Alex Wright and Peter Zograf for many helpful discussions and comments. Finally we thank the anonymous referee for numerous comments that led to an improvement of the presentation and the proof.

\end{document}